\documentclass[12pt]{amsart}
\usepackage{amssymb}
\usepackage[all]{xy}

\usepackage{graphicx}
\numberwithin{equation}{section}

\advance\headheight 2pt
\calclayout
\allowdisplaybreaks[3]

\theoremstyle{plain}
\newtheorem{prop}{Proposition}

\newtheorem{theo}[prop]{Theorem}
\newtheorem{coro}[prop]{Corollary}

\newtheorem{lemm}[prop]{Lemma}

\theoremstyle{definition}
\newtheorem{defn}[prop]{Definition}

\newtheorem{conj}[prop]{Conjecture}

\newtheorem{rema}[prop]{Remark}

\def\Fr{{\rm Fr}}
\def\tr{{\rm tr}}

\def\Br{{\rm Br}}
\def\Gal{{\rm G}}
\def\G{{\mathcal G}}   % abel. pro-l-Galois group 
\def\D{{\mathcal D}}   % decomposition group
\def\I{{\mathcal I}}   % inertia groups
\def\K{{\mathrm K}}

\def\pic{{\varphi}}
\def\Ga{{\Gamma}}%\value group
\def\Val{{\mathcal V}}
\def\DVal{{\mathcal D}{\mathcal V}}

\def\KK{\boldsymbol{K}}

\def\Aut{{\rm Aut}}

\def\Syl{{\rm Syl}}
\def\Frob{{\rm Fr}}

\def\no{\noindent}
\def\rk{{\rm rk}}
\newcommand{\trdeg}{{\rm tr}\, {\rm deg}}

\def\cH{{\mathcal H}}

\def\cT{{\mathcal T}}

\def\rK{{\mathrm  K}}
\def\rH{{\mathrm  H}}

\def\dv{{\rm div}}
\newcommand{\Hom}{{\rm Hom}}
\newcommand{\Ker}{{\rm Ker}}

\def\CH{{\rm CH}}

\def\lra{\longrightarrow}
\def\ra{\rightarrow}

\def\A{{\mathbb A}}
\def\C{{\mathbb C}}

\def\F{{\mathbb F}}

\def\P{{\mathbb P}}
\def\Q{{\mathbb Q}}

\def\Z{{\mathbb Z}}
\def\C{{\mathbb C}}
\def\N{{\mathbb N}}         
\def\R{{\mathbb R}}

\def\sc{{\mathsf c}}

\def\ml{{\mathfrak l}}

\def\mL{{\mathfrak L}}

\def\Pic{{\rm Pic}}
\def\Div{{\rm Div}}
\def\SL{{\rm SL}}
\def\GL{{\rm GL}}
\def\PGL{{\rm PGL}}
\def\Sym{{\rm Sym}}

\def\H{{\mathrm H}}

\def\pr{{\rm pr}}

\def\char{{\rm char}}

\def\rB{{\rm B}}

\def\rH{{\rm H}}

\def\rR{{\rm R}}

\begin{document}

\author{Fedor Bogomolov}
\address{Courant Institute of Mathematical Sciences, N.Y.U. \\
 251 Mercer str. \\
 New York, NY 10012, U.S.A.}
\email{bogomolo@cims.nyu.edu}

\author{Yuri Tschinkel}
\address{Courant Institute of Mathematical Sciences, N.Y.U. \\
 251 Mercer str. \\
 New York, NY 10012, U.S.A.}
\email{tschinkel@cims.nyu.edu}


\title[Birational Anabelian Geometry]{Introduction to birational anabelian geometry}

\begin{abstract}
We survey recent developments in the Birational Anabelian Geometry program
aimed at the reconstruction of function fields of 
algebraic varieties over algebraically closed fields
from pieces of their absolute Galois groups. 
\end{abstract}

\date{\today}

\maketitle
\tableofcontents

\setcounter{section}{0}       
\section*{Introduction}
\label{sect:introduction}

The essence of Galois theory is to {\em lose} information, by passing from 
a field $k$, an algebraic structure with two compatible operations,
to a (profinite) group, its absolute Galois group $G_k$ or 
some of its quotients.  The original goal of testing solvability in radicals 
of polynomial equations in one variable over the rationals 
was superseded by the study of deeper connections between the arithmetic in 
$k$, its ring of integers, and its completions with respect to various
valuations on the one hand, 
and (continuous) representations of $G_k$ on the other hand. 
The discovered structures turned out to be extremely rich, and 
the effort led to the development of deep and fruitful theories:
class field theory (the study of abelian extensions of $k$) and 
its nonabelian generalizations, the Langlands program. 
In fact, techniques from class field theory (Brauer groups)
allowed one to deduce that Galois groups of global fields encode the field:

\begin{theo}[Neukirch-Uchida \cite{neukirch}, \cite{uchida}] 
\label{thm:neukirch}
Let $K$ and $L$ be number fields or 
function fields of curves over finite fields with 
isomorphic Galois groups 
$$
G_{K^{\rm solv}/K}\simeq G_{L^{\rm solv}/L}
$$
of their maximal solvable extensions. 
Then 
$$
L\simeq K.
$$
\end{theo}

In another, more geometric direction, 
Galois theory was subsumed in the theory of the \'etale
fundamental group.  Let $X$ be an algebraic variety over a field $k$. Fix an algebraic closure $\bar{k}/k$ and let $K=k(X)$ be the function field of $X$. 
We have an associated exact sequence
\begin{equation}
%\label{eqn:basic}
1\ra \pi_1(X_{\bar{k}})\ra \pi_1(X)\stackrel{\pr_X}{\lra} G_k\ra 1  \tag{$\Psi_X$}
\end{equation}
of \'etale fundamental groups, exhibiting an action of 
the Galois group of the ground field $k$ on the {\em geometric fundamental group} 
$\pi_1(X_{\bar{k}})$. Similarly, we have an exact sequence of Galois groups
\begin{equation}
%\label{eqn:basic1}
1\ra G_{\bar{k}(X)}\ra G_{K}\stackrel{\pr_{K}}{\lra} G_k\ra 1. \tag{$\Psi_K$}
\end{equation}
Each $k$-rational point on $X$ gives rise to a section of $\pr_X$ and $\pr_K$. 

When $X$ is a smooth projective curve of genus $\mathsf g\ge 2$, 
its geometric fundamental group $\pi_1(X_{\bar{k}})$ is a profinite group in 
$2\mathsf g$ generators subject to one relation. 
Over fields of characteristic zero, 
these groups depend only on $\mathsf g$ but not
on the curve.  
However, the sequence  ($\Psi_X$) gives rise to a plethora of representations of $G_k$ and 
the resulting configuration is so strongly 
rigid\footnote{``ausserordentlich stark'', as Grothendieck puts it in \cite{groth-letter}}   
that it is natural to expect that it encodes 
much of the geometry and arithmetic of $X$ over $k$.  

For example, let $k$ be a 
{\em finite field} and $X$ an abelian variety over $k$ of dimension $\mathsf g$.
Then $G_k$ is the procyclic group $\hat{\Z}$, generated by the Frobenius,  
which acts on the Tate module 
$$
T_{\ell}(X)=\pi^a_{1,\ell}(X_{\bar{k}})\simeq \Z_{\ell}^{2\mathsf g},
$$
where $\pi^a_{1,\ell}(X_{\bar{k}})$ is the $\ell$-adic quotient of the abelianization
$\pi^a_{1}(X_{\bar{k}})$ of the
\'etale fundamental group.
By a theorem of Tate~\cite{tate}, 
the characteristic polynomial of the Frobenius determines $X$, up to isogeny. 
Moreover, if $X$ and $Y$ are abelian varieties over $k$ then
$$
\Hom_{G_k}(T_{\ell}(X), T_{\ell}(Y)) \simeq \Hom_k(X,Y)\otimes \Z_{\ell}.   
$$
Similarly, if $k$ is a {\em number field} and $X,Y$ abelian varieties over $k$ then 
$$
\Hom_{G_k}(\pi_1^a(X), \pi_1^a(Y)) \simeq \Hom_k(X,Y)\otimes \hat{\Z},
$$
by a theorem of Faltings~\cite{faltings-tate}.

\

With these results at hand, Grothendieck conjectured in \cite{groth-letter} 
that there is a certain class of {\em anabelian} varieties, defined over a field
$k$ (which is finitely generated over its prime field), which are
{\em characterized} by their fundamental groups. 
Main candidates are hyperbolic curves and varieties which can be successively
fibered by hyperbolic curves.
%over fields of characteristic zero (``vorsichtshalber'').   
There are three related conjectures:

\

\noindent
{\bf Isom:} An anabelian variety $X$ is determined by $(\Psi_X)$, i.e., by
the profinite group $\pi_1(X)$ together with the action of $G_k$.  
%and the exact sequence \eqref{eqn:basic}. 

\ 

\noindent
{\bf Hom:} If $X$ and $Y$ are anabelian, then 
there is a bijection 
$$
\Hom_k(X,Y) = \Hom_{G_k}(\pi_1(X),\pi_1(Y))/\sim
$$
between the set of dominant $k$-morphisms and $G_k$-equivariant
open homomorphisms of fundamental groups, modulo conjugacy (inner automorphisms by 
the geometric fundamental group of $Y$).

\

\noindent
{\bf Sections:} If $X$ is anabelian then there is a bijection between 
the set of rational points $X(k)$ and the set of 
sections of $\pr_X$ (modulo conjugacy).    

\

Similar conjectures can be made for nonproper varieties. Excising points from 
curves makes them ``more'' hyperbolic. Thus, one may reduce to the generic point of
$X$, replacing the fundamental group by the Galois group of the function field $K=k(X)$. 
In the resulting {\em birational} version of Grothendieck's conjectures, 
the exact sequence ($\Psi_X$) is replaced by
($\Psi_K$) and the projection $\pr_X$ by $\pr_K$. 

\

These conjectures have generated wide interest and stimulated intense research.  
Here are some of the highlights of these efforts:
\begin{itemize}
\item proof of the birational {\bf Isom}-conjecture for function fields over
$k$, where $k$ is finitely generated over its prime field, by Pop \cite{pop};
\item proof of the birational {\bf Hom}-conjecture over sub-$p$-adic fields
$k$, i.e., $k$ which are contained in a finitely generated extension of $\Q_p$,
by Mochizuki \cite{mochizuki};
\item 
proof of the birational {\bf Section}-conjecture for local fields of characteristic zero, by 
K\"onigsmann \cite{koenigsmann}.
\end{itemize}
Here is an incomplete list of other significant result in this area  \cite{N}, \cite{V2},
\cite{V91}, \cite{T}. 
In all cases, the proofs relied on {\em nonabelian} properties in the structure of 
the Galois group $G_K$, respectively, the relative Galois group. 
Some of these developments were surveyed in \cite{ihara-naka}, \cite{faltings}, 
\cite{Na-Mo}, \cite{pop-alter}, \cite{P-2}, \cite{mochi-topics}. 

After the work of Iwasawa the study of representations
of the maximal pro-$\ell$-quotient $\G_K$ of the absolute Galois group $G_K$ 
developed into a major branch of number theory and geometry. So it was natural to 
turn to pro-$\ell$-versions of the hyperbolic anabelian conjectures, replacing 
the fundamental groups by their maximal pro-$\ell$-quotients and 
the absolute Galois group $G_K$ by $\G_K$. 
Several results in this direction were obtained in  \cite{corry-pop}, \cite{saidi-tamagawa}. 

\

A very different intuition evolved from higher-dimensional 
birational algebraic geometry. One of the basic questions in this area
is the characterization of fields isomorphic to purely transcendental extensions of
the ground field, i.e., varieties birational to projective space. Interesting examples
of function fields arise from faithful representations of {\em finite} groups
$$
G \ra  \Aut(V),
$$
where $V=\A^n_k$ is the standard affine space over $k$.
The corresponding variety
$$
X=V/G
$$  
is clearly unirational. When $n\le 2$ and $k$ is algebraically closed 
the quotient is rational (even though there exist unirational but nonrational surfaces
in positive characteristic). The quotient is also rational when $G$ is abelian 
and $k$ algebraically closed. 

Noether's problem (inspired by invariant 
theory and the inverse problem in Galois theory) 
asks whether or not $X=V/G$ is rational for nonabelian groups.
The first counterexamples were constructed by Saltman \cite{saltman}. 
Geometrically, they are quotients of products of projective spaces by
projective actions of finite {\em abelian groups}.
The first obstruction to (retract) rationality was described in terms of 
Azumaya algebras and the {\em unramified} Brauer group
$$
\Br_{nr}(k(X))=\rH^2_{nr}(X),
$$
(see Section~\ref{sect:group}). 
A group cohomological interpretation of these examples 
was given by the first author in \cite{bog-87}; it allowed one to generate many other examples
and elucidated the key structural properties of the obstruction group. 
This obstruction can be computed in terms of $G$, in particular, it does not
depend on the chosen representation $V$ of $G$:
\begin{equation*}
\label{eqn:br}
\rB_0(G):=\Ker\left( \rH^2(G,\Q/\Z)\ra \prod_{B} \rH^2(B,\Q/\Z) \right),
\end{equation*}
where the product ranges over the set of subgroups $B\subset G$ which are 
generated by {\em two commuting} elements. 
A key fact is that, for $X=V/G$,  
$$
\rB_0(G)= \Br_{nr}(k(X))=\rH^2_{nr}(X),
$$
see Section~\ref{sect:group} and Theorem~\ref{thm:b0g}. 

\

One has a decomposition into primary components
\begin{equation} 
\rB_0(G)=\oplus_{\ell}\,\, \rB_{0,\ell}(G),
\end{equation}
and computation of each piece reduces
to computations of cohomology of the $\ell$-Sylow subgroups of $G$,
with coefficients in $\Q_{\ell}/\Z_{\ell}$.

\

We now restrict to this case, i.e., finite $\ell$-groups $G$ and $\Q_{\ell}/\Z_{\ell}$-coefficients.  
Consider the exact sequence
$$
1\ra Z\ra G^c\ra G^a\ra 1,
$$
where 
$$
G^c=G/[[G,G],G]
$$
is the canonical central extension of the abelianization 
$$
G^a=G/[G,G].
$$
We have
\begin{equation}
\label{eqn:bgg}
\rB_0(G^c)\hookrightarrow \rB_0(G)
\end{equation}
(see Section~\ref{sect:group}); in general, the image is a {\em proper} subgroup.  
The computation of $\rB_0(G^c)$ is
a problem in linear algebra:  
We have a well-defined map 
(from ``skew-symmetric matrices'' on $G^a$, 
considered as a linear space over $\Z/{\ell}$)  to the center of $G^c$: 
$$
\begin{array}{ccc}
\wedge^2(G^a)  & \stackrel{\lambda}{\lra} &  Z\\
(\gamma_1,\gamma_2)& \mapsto & [\tilde{\gamma}_1,\tilde{\gamma}_2],
\end{array}
$$ 
where $\tilde{\gamma}$ is some lift of $\gamma\in G^a$ to $G^c$.
Let 
$$
\rR(G^c):=\Ker(\lambda)
$$ 
be the subgroup of relations in $\wedge^2(G^a)$ 
(the subgroup generated by ``matrices'' of rank one).   
We say that $\gamma_1,\gamma_2$ form a {\em commuting pair}
if 
$$
[\tilde{\gamma}_1,\tilde{\gamma}_2]=1\in Z.
$$
Let 
$$
\rR_{\wedge}(G^c):=\langle \gamma_1\wedge \gamma_2\rangle \subset \rR(G^c)
$$ 
be the subgroup generated by commuting pairs. 
The first author proved in \cite{bog-87} that 
$$
\rB_0(G^c)=\left(\rR(G^c)/\rR_{\wedge}(G^c)\right)^{\vee}.
$$
Using this representation it is easy to produce examples
with nonvanishing $\rB_{0}(G)$, thus nonrational fields of $G$-invariants, 
already for central extensions of $(\Z/\ell)^4$ by $(\Z/\ell)^3$ \cite{bog-87}.

\

Note that for $K=k(V)^{G}$ the group $G$ is naturally a quotient of the absolute Galois group $G_K$. 
The sketched arguments from group cohomology suggested to focus on 
$\G_K$, the pro-$\ell$-quotient of $G_K$ and 
the pro-$\ell$-cohomology groups introduced above.
The theory of {\em commuting pairs} explained in Section~\ref{sect:valuations} 
implies that the groups $\G_K$ are very special: 
for any  function field $K$ over an algebraically closed field one has
$$
\rB_{0,\ell}(G_K)= \rB_0(\G_K)=\rB_0(\G_K^c).
$$

\

This lead to a {\em dismantling} of nonabelian aspects of 
anabelian geometry. For example, from this point of view it is unnecessary to assume 
that the Galois group of the ground field $k$ is large. On the contrary, 
it is preferable if $k$ is algebraically closed, or at least contains all $\ell^n$-th roots of 1. 
More significantly, while the {\em hyperbolic anabelian geometry} has dealt primarily with curves $C$, 
the corresponding $\rB_0(\G_{k(C)})$, and hence $\rB_0(\G^c_{k(C)})$,  are {\em trivial}, 
since the $\ell$-Sylow subgroups of $G_{k(C)}$ are free. 
Thus we need to consider function 
fields $K$ of transcendence degree at least 2 over $k$. 
It became apparent that in these cases, at least over $k=\bar{\F}_p$, 
$$
\rB_0(\G_K^c) = \rH^2_{nr}(k(X))
$$
encodes a wealth of information about $k(X)$. In particular, it determines all higher unramified
cohomological invariants of $X$ (see Section~\ref{sect:bloch-kato}).

Let $p$ and $\ell$ be distinct primes and $k=\bar{\mathbb F}_p$ an algebraic closure of $\mathbb F_p$.
Let  $X$ be an algebraic variety over $k$ and $K=k(X)$ its function field 
($X$ will be called a {\em model} of $K$).   
In this situation, $\G_K^a$ is a torsion-free $\Z_{\ell}$-module. 
Let $\Sigma_K$ be 
the set of not procyclic subgroups of $\G_K^a$ 
which lift to abelian subgroups in the canonical 
central extension 
$$
\G_K^c=\G_K/[[\G_K,\G_K],\G_K] \ra \G^a_K. 
$$
The set $\Sigma_K$ is canonically encoded in 
$$
\rR_{\wedge}(\G^c_K)\subset \wedge^2(\G^a_K),
$$ 
a group that carries {\em less} information than $\G^c_K$ (see Section~\ref{sect:wishful}).

\

The main goal of this survey is to explain the background
of the following result,  proved in \cite{bt0} and \cite{bt1}:

\begin{theo}                                    
\label{thm:0}
Let $K$ and $L$ be function fields over
algebraic closures of finite fields $k$ and $l$,  
of characteristic $\neq \ell$. 
Assume that the transcendence degree of $K$ over $k$ is at least two and  
that there exists an isomorphism
\begin{equation}
\label{eqn:psi-dual}
\Psi=\Psi_{K,L}\,:\, \G^a_K\stackrel{\sim}{\lra} \G^a_{L}
\end{equation}
of abelian pro-$\ell$-groups inducing a bijection of sets
$$
\Sigma_K = \Sigma_{L}.
$$
Then $k=l$ and there exists a constant $\epsilon\in \Z_{\ell}^{\times}$
such that $\epsilon^{-1}\cdot\Psi$ is induced from 
a unique isomorphism of perfect closures
$$
\bar{\Psi}^*\,:\, \bar{L}\stackrel{\sim}{\lra} \bar{K}.
$$
\end{theo}

The intuition behind Theorem~\ref{thm:0} is that the arithmetic and geometry of 
varieties of transcendence degree $\ge 2$ over algebraically closed ground fields 
is governed by {\em abelian} or {\em almost abelian} phenomena. One of the consequences
is that central extensions of abelian groups provide {\em universal} counterexamples
to Noether's problem, and more generally, provide {\em all} finite cohomological obstructions
to rationality, at least over $\bar{\F}_p$ (see Section~\ref{sect:bloch-kato}).

\

Conceptually, the proof of Theorem~\ref{thm:0} explores a {\em skew-symmetric} incarnation of the field, 
which is a {\em symmetric object}, with two symmetric operations.  
Indeed, by Kummer theory, we can identify
$$
\G^a_K = \Hom(K^{\times}/k^{\times},\Z_{\ell}). 
$$
Dualizing again, we obtain
$$
\Hom(\G^a_K, \Z_{\ell}) = \hat{K}^{\times},
$$
the pro-$\ell$-completion of the multiplicative group of $K$.
Recall that 
$$
K^{\times}=\rK_1^M(K),
$$
the first Milnor $\rK$-group of the field. 
The elements of $\wedge^2(\G^a_K)$ are
matched with symbols in Milnor's $\rK$-group $\rK_2^M(K)$. 
The symbol 
$(f,g)$ is infinitely divisible 
in $\rK_2^M(K)$ if and only if $f,g$ are algebraically dependent, i.e., 
$f,g\in E=k(C)$ for some curve $C$ (in particular, we get no information when $\trdeg_k(K)=1$).
In Section~\ref{sect:k-theory} we describe how to reconstruct homomorphisms of fields
from compatible homomorphisms 

\centerline{
\xymatrix{
\rK_1^M(L)   \ar[r]^{\!\!\psi_1} & \rK_1^M(K), \\
\rK_2^M(L)   \ar[r]^{\!\!\psi_2} & \rK_2^M(K).
}
}

\

\noindent
Indeed, the multiplicative group of the ground field $k$ is characterized as the subgroup of 
infinitely divisible elements of $K^{\times}$, 
thus 
$$
\psi_1\colon \P(L)=L^{\times}/l^{\times}\ra \P(K)=K^{\times}/k^{\times},
$$ 
a homomorphism of multiplicative groups (which we assume to be injective).  
The compatibility with $\psi_2$ means that infinitely divisible 
symbols are mapped to infinitely divisible symbols, i.e.,   
$\psi_1$ maps multiplicative groups $F^{\times}$ of 1-dimensional subfields $F\subset L$ to 
$E^{\times}\subset K^{\times}$, for 1-dimensional $E\subset K$. This implies that 
already each $\P^1\subset \P(L)$ is mapped to a $\P^1\subset \P(K)$. 
The Fundamental theorem of projective geometry (see Theorem~\ref{thm:recon-fields})  
shows that (some rational power of) $\psi_1$ is a restriction of a 
homomorphisms of fields $L\ra K$.

Theorem~\ref{thm:0} is a pro-$\ell$-version of this result.
Kummer theory provides the isomorphism
$$
\Psi^* \colon \hat{L}^{\times}\ra \hat{K}^{\times}
$$
The main difficulty is to recover
the {\em lattice}
$$
K^{\times}/k^{\times}\otimes \Z_{(\ell)}^{\times} \subset \hat{K}^{\times}.
$$
This is done in several stages. 
First, the theory of {\em commuting pairs} (see \cite{BT}) 
allows to reconstruct abelianized 
inertia and decomposition groups of valuations
$$
\I^a_{\nu}\subset \D^a_{\nu}\subset \G^a_K.
$$ 
%every liftable subgroup $\sigma\in \Sigma_K$ contains an $\iota\in \I^a_{\nu}$, for some valuation $\nu$.  
Note that for divisorial valuations $\nu$ we have $\I^a_{\nu}\simeq \Z_{\ell}$, 
and the set
$$
\I^a=\{ \I^a_{\nu}\} 
$$
resembles a $\Z_{\ell}$-{\em fan} in $\G^a_{K}\simeq \Z_{\ell}^{\infty}$. 
The key issue is to pin down, canonically, a topological generator for each of these
$\I^a_{\nu}$. 
The next step is to show that 
$$
\Psi^*(F^{\times}/l^{\times})\subset \hat{E}^{\times}\subset \hat{K}^{\times}
$$
for some 1-dimensional $E\subset K$. This occupies most of the paper~\cite{bt0}, 
for function fields of surfaces. The higher-dimensional case, treated in \cite{bt1},
proceeds by induction on dimension.
The last step, i.e., matching of projective structures on multiplicative groups, 
is then identical to the arguments used above in the context of Milnor $\rK$-groups.

\

The Bloch--Kato conjecture says that $\G^c_K$ contains all information about 
the cohomology of $G_K$, with finite constant coefficients  
(see Section~\ref{sect:bloch-kato} for a detailed discussion).
Thus we can consider Theorem~\ref{thm:0} as a {\em homotopic} version of the Bloch--Kato 
conjecture, i.e., $\G_K^c$ determines the field $K$ itself, modulo purely-inseparable extension. 

\

{\em Almost abelian} anabelian geometry evolved
from the Galois-theoretic interpretation of Saltman's counterexamples described above  
and the Bloch--Kato conjecture.  These ideas, 
and the ``recognition'' technique used in the proof of Theorem~\ref{thm:0}, 
were put forward in \cite{bog-87}, \cite{bog-stable}, \cite{B-1}, \cite{B-2}, \cite{B-3}, 
and developed in \cite{BT}, \cite{bt0}, \cite{bt-milnor}, and \cite{bt1}. 
In recent years, this approach has attracted the attention of several experts, 
for example, F. Pop, see \cite{pop-unpub}, as well as his webpage, for other preprints on this topic, 
which contain his version of the recognition procedure of $K$ from $\G^c_K$, for the same class of fields $K$. 
Other notable contributions are due to Chebolu, Efrat, and Minac  \cite{efrat-minac}, \cite{chebolu-minac}.

\

Several ingredients of the the proof of Theorem ~\ref{thm:0} sketched above
appeared already in Grothendieck's anabelian geometry, relating the {\em full}  
absolute Galois group of function fields to the geometry of projective models. 
Specifically, even before Grothendieck's insight, it was understood by Uchida and Neukirch 
(in the context of number fields and function fields of curves over finite fields) 
that the identification of decomposition groups of 
valuations can be obtained in purely group-theoretic terms as, roughly speaking, 
subgroups with nontrivial center. 
Similarly, it was clear that Kummer theory essentially captures the multiplicative structure 
of the field and that the projective structure on $\P_k(K)$ encodes the additive structure. 
The main difference between our approach and the techniques of, e.g., Mochizuki \cite{mochizuki} 
and Pop \cite{pop-unpub} is the theory of commuting pairs which is based on an unexpected
coincidence: the {\em minimal} necessary condition for the commutation of two elements of the 
absolute Galois group of a function field $K$ 
is also sufficient and already implies that these elements belong 
to the same decomposition group. It suffices to check this condition on $\G^c_K$, which 
{\em linearizes} the commutation relation. 
Another important ingredient in our approach is the correspondence between {\em large} free quotients of
$\G^c_K$ and integrally closed 1-dimensional subfields of $K$. Unfortunately, in full generality, this
conjectural equivalence remains open (see the discussion in Section~\ref{sect:wishful}).
However, by exploiting geometric properties of projective models of $K$ we succeed
in proving it in many important cases, which suffices for solving the recognition problem and
for several other applications.  

\

Finally, in Section~\ref{sect:curves} we discuss almost abelian phenomena in 
Galois groups of {\em curves} which occur for competely different reasons. 
An application of a recent result of Corvaja--Zannier concerning the divisibility of
values of recurrence sequences leads to a Galois-theoretic 
Torelli-type result for curves over finite fields. 

\

\no
{\bf Acknowledgments.} 
We have benefited from conversations with J.-L. Colliot-Th\'el\`ene,  B. Hassett, and M. Rovinsky.
We are grateful to the referee for helpful remarks and suggestions. 
The first author was partially supported  by NSF grant DMS-0701578. 
The second author was partially supported by NSF grants DMS-0739380 and 
0901777.

\section{Abstract projective geometry}
\label{sect:projective}

\begin{defn}
\label{defi:proj}
A \emph{projective structure} is a pair $(S,\mL)$ where 
$S$ is a set (of points) and $\mL$ a collection of subsets 
${\mathfrak l}\subset S$ (lines) such that
\begin{itemize}
\item[P1] there exist an $s\in S$ and an $\mathfrak l\in \mL$
such that $s\notin \mathfrak l$;
\item[P2] for every $\mathfrak l\in \mL$ there exist at least
three distinct $s,s',s''\in \mathfrak l$;
\item[P3] for every pair of distinct $s,s'\in S$ there exists exactly 
one 
$$
\mathfrak l={\mathfrak l}(s,s')\in \mL
$$ 
such that $s,s'\in \mathfrak l$;
\item[P4] for every quadruple of pairwise distinct
$s,s',t,t'\in S$ one has
$$
{\mathfrak l}(s,s')\cap {\mathfrak l}(t,t')\neq \emptyset\,\, \Rightarrow\,\, 
{\mathfrak l}(s,t)\cap {\mathfrak l}(s',t')\neq \emptyset  . 
$$ 
\end{itemize}
\end{defn}

In this context, one can define (inductively) the dimension of a projective space:
a two-dimensional projective space, i.e., 
a projective plane, is the set of points on lines passing 
through a line and a point outside this line; a three-dimensional space is
the set of points on lines passing through a plane and a point outside this plane, etc.  
 
A \emph{morphism} of projective structures 
$\rho\,:\, (S,\mL)\ra (S',\mL')$ 
is a map of sets $\rho\,:\, S\rightarrow S'$ preserving lines, i.e., 
$\rho(\ml)\in \mL'$, for all $\ml\in \mL$.

\label{defn:pappus}
A projective structure $(S,\mL)$ satisfies {\em Pappus' axiom}
if 
\begin{itemize}
\item[PA]
for all 2-dimensional subspaces and every configuration
of six points and lines in these subspaces as below

\

\centerline{\includegraphics[width=.4\textwidth]{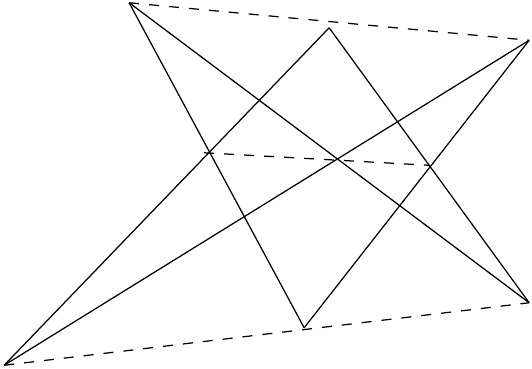}}

\

\noindent
the intersections are collinear. 
\end{itemize}

\

The following Fundamental theorem of abstract projective geometry
goes back at least to Schur and Hessenberg, but there were many  
researchers before and after exploring the various interconnections
between different sets of axioms 
(Poncelet, Steiner, von Staudt, Klein, Pasch, Pieri, Hilbert, and others).\footnote{{\it But there is one group of deductions 
which cannot be ignored in any consideration of the 
principles of Projective Geometry. 
I refer to the theorems, by which it is proved that numerical coordinates, 
with the usual properties, can be defined without the introduction of distance as a fundamental idea. The establishment of this result is one of the triumphs of modern mathematical thought.} A.N. Whitehead, 
``The axioms of projective geometry'', p. v, 1906.}

\begin{theo}[Reconstruction]
Let $(S,\mL)$ be a projective structure of dimension $n\ge 2$ 
which satisfies Pappus' axiom. 
Then there exists a vector space $V$ over a field $k$ and 
an isomorphism 
$$
\sigma\,:\, \P_k(V)\stackrel{\sim}{\longrightarrow} S.
$$
Moreover, for any two such triples $(V,k,\sigma)$ and $(V',k', \sigma')$
there is an isomorphism 
$$
V/k\stackrel{\sim}{\longrightarrow}V'/k'
$$
compatible with $\sigma,\sigma'$ and unique up to 
homothety $v\mapsto \lambda v$, $\lambda\in k^{\times}$.
\end{theo}

Main examples are of course the sets of $k$-rational points of 
the usual projective $\P^n$ 
space over $k$ of dimension $n\ge 2$. 
Then $\P^n(k)$ carries a projective structure: lines
are the usual projective lines $\P^1(k)\subset \P^n(k)$.

A related example arises as follows:
Let $K/k$ be an extension of fields. Then
$$
S:=\P_k(K)=(K\setminus 0)/k^{\times}
$$ 
carries a natural (possibly, infinite-dimensional) 
projective structure. Moreover, the
\emph{multiplication} in $K^{\times}/k^{\times}$ preserves this structure.  
In this setup we have the following reconstruction theorem (\cite[Theorem 3.6]{bt0}):

\begin{theo}[Reconstructing fields] \label{thm:recon-fields}
Let $K/k$ and $K'/k'$ be field extensions of degree $\ge 3$ and 
$$
\bar{\psi}\,:\, S=\P_k(K)\ra \P_{k'}(K')=S'
$$ 
an injective homomorphism of abelian groups compatible with 
projective structures. 
Then $k\simeq k'$ and $K$ is isomorphic to a subfield of $K'$.  
\end{theo}

The following strengthening is due to M. Rovinsky. 

\begin{theo}
Let $S$ be an abelian group equipped with a compatible structure of a projective space. 
Then there exist fields $k$ and $K$ such that $S=\P_k(K)$. 
\end{theo}

\begin{proof}
There is an embedding of $S=\P(V)$ as a projective subspace into
$\PGL(V)$. Its preimage in $\GL(V)$ is a linear subspace minus
a point. Since $V$ is invariant under products (because
$\P(V)$ is) we obtain that $V$ is a commutative subalgebra
of ${\rm Mat}(V)$ and every element is invertible - hence it is a field.
\end{proof}

%\section{Pregeometries and geometries}

\

Related {\em reconstruction theorems} of ``large'' fields  have recently emerged in model theory. 
The setup there is as follows:
A \emph{combinatorial pregeometry} (\emph{finitary matroid}) is a 
pair $(\mathcal P,cl)$ where $\mathcal P$ is a set and 
$$
cl \,:\, {\rm Subsets}(\mathcal P) \rightarrow {\rm Subsets}(\mathcal P), 
$$
such that for all $a, b \in \mathcal P$  and all $Y, Z \subseteq \mathcal P$ one has:
\begin{itemize}
\item $Y \subseteq cl(Y)$,
\item if $Y \subseteq Z$, then $cl(Y) \subseteq cl(Z)$,
\item $cl(cl(Y)) = cl(Y)$,
\item if $a \in cl(Y)$, then there is a finite subset $Y'\subset  Y$ such that $a \in cl(Y')$ (finite character), 
\item (exchange condition) 
if $a \in cl( Y \cup \{ b\}) \setminus cl(Y)$, 
then $b \in cl(Y\cup \{ a\})$.
\end{itemize}

A \emph{geometry} is a pregeometry such that $cl({a}) = {a}$, for all $a \in \mathcal P$, and 
$cl(\emptyset) = \emptyset$. Standard examples are provided by:
\begin{enumerate} 
\item $\mathcal P=V/k$, a vector space over a field 
$k$ and $cl(Y)$ the $k$-span of $Y\subset \mathcal P$;
\item $\mathcal P=\mathbb P_k(V)$, the usual \emph{projective} space
over a field $k$;
\item $\mathcal P= \mathcal P_k(K)$, 
a field $K$ containing an 
algebraically closed subfield $k$ and 
$cl(Y)$ - the normal closure of $k(Y)$ in $K$, 
note that a \emph{geometry} is obtained after factoring by $x\sim y$  iff $cl(x)=cl(y)$. 
\end{enumerate}

It turns out that a  sufficiently large field can reconstructed from the {\em geometry} of its 1-dimensional
subfields. 

\begin{theo}[Evans--Hrushovski \cite{eh1},\cite{eh2} / Gismatullin \cite{gism}]
Let $k$ and $k'$ be algebraically closed fields,  
$K/k$ and $K'/k'$ field extensions of transcendence degree $\ge 5$ over $k$, 
resp. $k'$. Then, every 
isomorphism of combinatorial geometries
$$
\mathcal P_{k}(K) \ra \mathcal P_{k'}(K') 
$$
is induced by an isomorphism of purely inseparable closures
$$
\bar{K} \rightarrow \overline{K'}. 
$$
\end{theo}

In the next section, we show how to reconstruct a field  of transcendence degree $\ge 2$ 
from its {\em projectivized} multiplicative group and 
the ``geometry'' of multiplicative groups of 1-dimensional subfields.

\

\section{K-theory}
\label{sect:k-theory}

Let $\rK^{M}_i(K)$ be $i$-th
Milnor K-group of a field $K$. Recall that
$$
\rK^M_1(K)=K^{\times}
$$ 
and that there is a canonical surjective homomorphism
$$
\sigma_K\,:\, \rK^M_1(K)\otimes \rK^M_1(K)\ra \rK_2^M(K);
$$
we write $(x,y)$ for the image of $x\otimes y$.
The kernel of $\sigma_K$ 
is generated by symbols $x\otimes (1-x)$, 
for $x\in K^{\times}\setminus 1$. 
Put 
$$
\bar{\rK}^M_i(K):= \rK^M_i(K)/{\rm infinitely  \,\, divisible\,\, elements}, \quad i=1,2.
$$

\begin{theo}\cite{bt-milnor}
Let $K$ and $L$ be function fields of transcendence degree $\ge 2$ 
over an algebraically closed field $k$, resp. $l$.  
Let
\begin{equation*}
\label{eqn:psi*}
\bar{\psi}_1\,:\, \bar{\rK}^M_1(K)\ra \bar{\rK}^M_1(L)
\end{equation*}
be an injective homomorphism. 

Assume that there is a commutative diagram 

\

\centerline{
\xymatrix{
{\bar{\rK}^M_1(K)}\otimes {\bar{\rK}^M_1(K)} \ar[rr]^{\bar{\psi}_1\otimes \bar{\psi}_1} 
\ar[d]_{\sigma_K}  && 
\ar[d]^{\sigma_L} \bar{\rK}^M_1(L)\otimes \bar{\rK}^M_1(L) \\
    \bar{\rK}^M_2(K)              \ar[rr]_{\bar{\psi}_2} && \bar{\rK}^M_2(L). 
}
}

\noindent
Assume that $\bar{\psi}_1(K^{\times}/k^{\times}) \not\subseteq E^{\times}/l^{\times}$, for
a 1-dimensional field $E\subset L$ (i.e., a field 
of transcendence degree 1 over $l$).  

Then there exist an $m\in \Q$  and a homomorphism of fields
$$
\psi\,:\, K\ra L
$$
such that the induced map on $K^{\times}/k^{\times}$ coincides with 
$\bar{\psi}_1^m$.
\end{theo}

\begin{proof}[Sketch of proof]
First we reconstruct the multiplicative group of the ground field as the subgroup of 
infinitely divisible elements:
An element $f\in K^{\times}={\rm K}_1^M(K)$ is infinitely divisible if and only if
$f\in k^{\times}$.  
In particular, 
$$
\bar{\rK}^M_1(K)=K^{\times}/k^{\times}.
$$
Next, we characterize multiplicative groups of 1-dimensional subfields:
Given a nonconstant $f_1\in K^{\times}/k^{\times}$, we have 
$$
\Ker_2(f_1)=E^{\times}/k^{\times},
$$ 
where $E=\overline{k(f_1)}^K$ is the normal closure in $K$ 
of the 1-dimensional field 
generated by $f_1$ and 
$$
\Ker_2(f):=\{\, g\in K^{\times}/k^{\times} = \bar{\rK}^M_1(K)\,\mid \, 
(f,g)= 0 \in \bar{\rK}_2^M(K) \,\}.
$$

At this stage we know the infinite-dimensional projective subspaces $\P(E)\subset \P(K)$. 
To apply Theorem~\ref{thm:recon-fields} we need to show that projective {\em lines} 
$\P^1\subset \P(K)$ are mapped to projective lines in $\P(L)$. 
It turns out that lines can be characterized as intersections of 
(shifted) $\P(E)$, for appropriate
1-dimensional $E\subset K$. 
The following technical result lies at the heart of the proof.

%Equivalently, the points of the lines $\P^1\subset \P(K)$ satisfy certain functional equations
%which are preserved under $\bar{\psi}_1$.
%The following Proposition~\ref{prop:funct}
%exhibits the lines $\P_k(\langle x,1\rangle)\subset \P_k(k(x))$ 
%as solutions of functional equations expressing nontriviality of intersections. 
\end{proof}

\begin{prop} \cite[Theorem 22]{bt-milnor}
\label{prop:funct}
Let $k$ be an algebraically closed field,
$K$ be an algebraically closed field extension of $k$,
$x,y\in K$ algebraically independent over $k$, 
$p\in\overline{k(x)}^{\times} \smallsetminus k\cdot x^{{\mathbb Q}}$ and
$q\in\overline{k(y)}^{\times}\smallsetminus k\cdot y^{{\mathbb Q}}$.
Suppose that 
$$
\overline{k(x/y)}^{\times}\cdot y\cap \overline{k(p/q)}^{\times}\cdot q \neq \emptyset.
$$
Then there exist 
\begin{enumerate}
\item  an $a\in \Q$,
\item $c_1,c_2\in k^{\times}$
such that
$$
p\in k^{\times}\cdot(x^a-c_1)^{1/a}, \quad 
q\in k^{\times}\cdot(y^a-c_2)^{1/a}
$$ 
and
$$
\overline{k(x/y)}^{\times}\cdot y\cap\overline{k(p/q)}^{\times}\cdot
q=k\cdot(x^a-cy^a)^{1/a},
$$
where $c=c_1/c_2$.
\end{enumerate}
\end{prop}

\begin{proof}
The following proof, which works in characteristic zero, has been suggested by M. Rovinsky
(the general case in \cite{bt-milnor} is more involved).

Assume that there is a nontrivial
$$
I\in \overline{k(x/y)}^{\times}\cdot y\cap\overline{k(p/q)}^{\times}\cdot q.
$$
We obtain equalities in $\Omega_{K/k}$:
\begin{equation}\label{ur-na-I}
\frac{\mathrm d(I/y)}{I/y}=r\cdot\frac{\mathrm d(x/y)}{x/y}
\quad \text{ and } \quad 
\frac{\mathrm d(I/q)}{I/q}=s\cdot\frac{\mathrm d(p/q)}{p/q},
\end{equation}
for some 
$$
r\in\overline{k(x/y)}^{\times}, \quad \text{  and  } \quad 
s\in\overline{k(p/q)}^{\times}.
$$
Using the first equation, rewrite  the second as 
$$
r\cdot\frac{\mathrm d(x/y)}{x/y}+\frac{\mathrm d(y/q)}{y/q}=s\cdot\frac{\mathrm d(p/q)}{p/q},
$$
or
$$
r\frac{\mathrm dx}{x}-s\frac{\mathrm dp}{p}=
r\cdot\frac{\mathrm dy}{y}+\frac{\mathrm d(q/y)}{q/y}-s\frac{\mathrm dq}{q}.
$$
The differentials on the left and on the right are linearly independent,
thus both are zero, i.e., $r=sf=sg-g+1$,
where 
$$
f=xp'/p\in\overline{k(x)}^{\times}\quad  \text{ and  } \quad g=yq'/q\in\overline{k(y)}^{\times},
$$
and $p'$ is derivative with respect to $x$, $q'$ the derivative 
In particular, $s=\frac{1-g}{f-g}$. Applying $\mathrm d\log$ to both sides,
we get 
$$
\frac{\mathrm ds}{s}=\frac{g'\mathrm dy}{g-1}+\frac{g'\mathrm dy-f'\mathrm dx}{f-g}
=\frac{f'}{g-f}\mathrm dx+\frac{g'(1-f)}{(1-g)(f-g)}\mathrm dy.
$$
As $\mathrm ds  / s$ is proportional to 
$$
\frac{\mathrm d(p/q)}{p/q}
=\frac{p'}{p}\mathrm dx-\frac{q'}{q}\mathrm dy=f\frac{\mathrm dx}{x}-g\frac{\mathrm dy}{y}dy,
$$
we get
$$
x\frac{f'}{f}=y\frac{g'(1-f)}{(1-g)g},
$$
$$
x\frac{f'}{(1-f)f}=y\frac{g'}{(1-g)g}.
$$
Note that the left side is in $\overline{k(x)}^{\times}$,
while the right hand side is in $\overline{k(y)}^{\times}$. It follows that
$$
x\frac{f'}{(1-f)f}=y\frac{g'}{(1-g)g}=a\in k.
$$
Solving the ordinary differential equation(s), we get
$$
\frac{f}{f-1}=c_1^{-1}x^a \quad \text{ and } \quad \frac{g}{g-1}=c_2^{-1}y^a
$$
for some $c_1,c_2\in k^{\times}$ and $a\in \mathbb Q$, so
$$
f=(1-c_1x^{-a})^{-1}=x\frac{\mathrm d}{\mathrm dx}\log(x^a-c_1)^{1/a},
$$
$$
g=(1-c_2y^{-a})^{-1}=y\frac{\mathrm d}{\mathrm dy}\log(y^a-c_2)^{1/a}.
$$
Thus finally,
$$
p=b_1\cdot(x^a-c_1)^{1/a}\quad  \text{and} \quad q=b_2\cdot(y^a-c_2)^{1/a}.
$$
We can now find 
$$
s=\frac{(1-c_1x^{-a})^{-1}c_2y^{-a}}{c_2y^{-a}-c_1x^{-a}}
=\frac{c_2(x^a-c_1)}{c_2x^a-c_1y^a}
$$ 
and then
$$
r=sf=\frac{c_2x^a}{c_2x^a-c_1y^a}=(1-c(x/y)^{-a})^{-1},
$$
where $c=c_1/c_2$.
From equation (\ref{ur-na-I}) we find
$$
\mathrm d\log(I/y)=-\frac{1}{a}\frac{\mathrm dT}{T(1-T)},
$$ 
where $T=c(x/y)^{-a}$,
and thus, 
$$
I=y\cdot b_3(1-c^{-1}(x/y)^a)^{1/a}=b_0(x^a-cy^a)^{1/a}.
$$ 
\end{proof}

This functional equation has the following projective interpretation:
If $E = k(x)$ then the image of each $\P^1\subset \P(E)$ under $\Psi$ 
lies in a rational normal curve given by (2) in Proposition~\ref{prop:funct}, where $a$ 
may {\em a priori} depend on $x$. However, a  
simple lemma shows that it is actually independent of $x$ (in characteristic zero), 
thus $\Psi^{1/a} $ extends to a field homomorphism. 
(In general, it is well-defined modulo powers of $p$, this brings up purely inseparable
extensions, which are handled by an independent argument.)

\section{Bloch-Kato conjecture}
\label{sect:bloch-kato}

Let $K$ be a field and $\ell$ a prime distinct from the characteristic of $K$.
Let
$$
\boldsymbol{\mu}_{\ell^n}:=\{ \sqrt[\ell^n]{1}\,\}
\,\,\,\text{ and }\,\,\,
\Z_{\ell}(1) =\lim_{\longleftarrow} \boldsymbol{\mu}_{\ell^n}.
$$
We will assume that $K$ contains all $\ell^n$-th roots of unity
and identify $\Z_{\ell}$ and $\Z_{\ell}(1)$.
Let $\G^a_K$ be the abelianization of the maximal 
pro-$\ell$-quotient of the absolute Galois group $G_K$. 

\begin{theo}[Kummer theory]
\label{thm:ga}
There is a canonical isomorphism
\begin{equation}
\label{eqn:kummer}
\rH^1(\Gal_K,\Z/\ell^n) =\rH^1(\G^a_K,\Z/\ell^n) = K^{\times}/\ell^n.
\end{equation}
\end{theo}

More precisely, the discrete group $K^{\times}/(K^{\times})^{\ell^n}$ and the compact profinite group $\G^a_K/\ell^n$ are
Pontryagin dual to each other, for a $\boldsymbol{\mu}_{\ell^n}$-duality, 
i.e., there is a perfect pairing 
$$
K^{\times}/(K^{\times})^{\ell^n} \times \G^a_K/\ell^n \ra \boldsymbol{\mu}_{\ell^n}.
$$
Explicitly, this is given by 
$$
(f,\gamma) \mapsto  \gamma(\sqrt[\ell^n]{f})/ \sqrt[\ell^n]{f} \in \boldsymbol{\mu}_{\ell^n}.
$$   
For $K=k(X)$, with $k$ algebraically closed of characteristic $\neq \ell$,
we have 
\begin{itemize}
\item $K^{\times}/k^{\times}$ is a free $\Z$-module and
$$
K^{\times} /(K^{\times})^{\ell^n} = (K^{\times}/k^{\times})/\ell^n, \quad  \text{ for all } \quad n\in \N;
$$
\item identifying $K^{\times}/k^{\times}\stackrel{\sim}{\longrightarrow} \Z^{(\mathrm I)}$, 
one has $K^{\times}/(K^{\times})^{\ell^n}\stackrel{\sim}{\longrightarrow}(\Z/\ell^n)^{(\mathrm I)}$
and 
$$
\G^a_K/\ell^n\stackrel{\sim}{\longrightarrow}(\Z/\ell^n(1))^{\mathrm I},
$$ 
in particular, the duality 
between $\hat{K}^{\times}=\widehat{K^{\times}/k^{\times}} $ and $\G^a_K$ is modeled on that between
$$
\{ \text{functions } \mathrm I\ra \Z_{\ell} \,\text{ tending to } 0 \, 
\text{ at } \,\infty \} \, \text{ and } \, \Z_{\ell}^{\mathrm I}.
$$
Since the index set $\mathrm I$ is not finite taking double-duals increases the space of 
{\em functions with finite support} to the space of 
{\em functions with support converging to zero}, i.e., 
the support modulo $\ell^n$ is finite, for all $n\in \N$. For function fields, 
the index set is essentially the set of irreducible divisors on a projective model of the field. 
This description is a key ingredient in the reconstruction of function fields from 
their Galois groups. 
\end{itemize}

In particular, an isomorphism of Galois groups 
$$
\Psi_{K,L} \colon \G^a_K\stackrel{\sim}{\lra} \G^a_L
$$
as in Theorem~\ref{thm:0} implies a canonical isomorphism 
$$
\Psi^*\colon \hat{K}^{\times} \simeq \hat{L}^{\times}.
$$

The Bloch--Kato conjecture, now a theorem established by Voevodsky \cite{MC-2}, \cite{MC-l}, 
with crucial
contributions by Rost and Weibel \cite{Weibel-1}, \cite{Weibel-2},   
describes the cohomology of the absolute 
Galois group $G_K$ through Milnor $\rK$-theory for all $n$:
\begin{equation}
\label{eqn:KK}
\rK_n^M(K)/\ell^n = \rH^n(G_K, \Z/\ell^n).
\end{equation}
There is an alternative formulation. 
Let $\G^c_K$ be the canonical central extension of $\G^a_K$ as in the Introduction. 
We have the diagram

\[
\centerline{
\xymatrix@!{
      & G_K\ar[dl]_{\,\,\,\pi_c} \ar[dr]^{\,\,\,\pi} & \\
 \G^c_K \ar[rr]_{\pi_a}  &   &           \G^a_K \\
}
}
\]

\

\begin{theo}
\label{thm:bkk}
The Bloch--Kato conjecture \eqref{eqn:KK} is equivalent to:
\begin{enumerate}
\item 
The map 
$$
\pi^* \colon \rH^*(\G^a_K, \Z/\ell^n)\to \rH^*(\Gal_K,\Z/\ell^n)
$$ 
is surjective and
\item 
$\
\Ker(\pi_a^*) = \Ker(\pi^*)$.  
\end{enumerate}
\end{theo}

\begin{proof}
The proof uses the first two cases of the
Bloch--Kato conjecture. 
The first is \eqref{eqn:kummer}, i.e., Kummer theory.
Recall that the cohomology ring of a torsion-free abelian group is the exterior 
algebra on $\rH^1$. We apply this to $\G^a_K$; combining with \eqref{eqn:kummer}
we obtain:
$$
\rH^*(\G_K^a,\Z/\ell^n)=\wedge^* (K^{\times}/\ell^n).
$$
Since $\G^c$ is a central extension of the torsion-free abelian group $\G^a_K$, 
the kernel of the ring homomorphism
$$
\pi_a^*  \colon \rH^*(\G^a_K, \Z/\ell^n)\ra  \rH^*(\G^c_K, \Z/\ell^n) 
$$
is an ideal $IH_K(n)$ generated by 
$$
\Ker\left(\rH^2(\G^a_K, \Z/\ell^n)\ra  \rH^2(\G^c_K, \Z/\ell^n)\right) 
$$
(as follows from the standard spectral sequence argument). We have an exact sequence 
$$
0\ra IH_K(n)\ra \wedge^*(K^{\times}/\ell^n) \ra \rH^*(\G^c,\Z/\ell^n).
$$
On the other hand, we have a diagram for the Milnor $\rK$-functor:

\

\centerline{
\xymatrix{
1\ar[r] & \tilde{I}_K(n) \ar[r] \ar@{>>}[d] &  \otimes^* (K^{\times}/\ell^n) \ar[r] \ar@{>>}[d]  & \rK^M_*(K)/\ell^n \ar@{=}[d] \ar[r] & 1\\
1\ar[r] & I_K(n) \ar[r] &  \wedge^* (K^{\times}/\ell^n) \ar@{=}[d] \ar[r] & \rK^M_*(K)/\ell^n \ar[r] & 1 \\
        &               & \rH^*(\G_K^a,\Z/\ell^n) &  
 }
}

\

\noindent
Thus the surjectivity of $\pi^*$ is equivalent to the surjectivity of 
$$
\rK^M_n(K)/\ell^n\ra \rH^n(\Gal_K,\Z/\ell^n).
$$ 

Part (2) is equivalent to 
$$
IH_K(n) \simeq I_K(n),
$$
under the isomorphism above. Both ideals are generated by degree 2 components. 
In degree 2, the claimed isomorphism follows from  
the Merkurjev--Suslin theorem
$$
\rH^2(\Gal_K,\Z/\ell^n) = \rK^M_2(K)/\ell^n.
$$
\end{proof}

Thus the Bloch--Kato conjecture implies that $\G^c_K$ 
completely captures the $\ell$-part of 
the cohomology of $\Gal_K$.
This led the first author to conjecture in  \cite{B-1} 
that the ``homotopy''
structure of $\Gal_K$ is also captured by $\G^c_K$ and that
morphisms between function fields $L\ra K$  
should be captured (up to purely inseparable extensions)
by morphisms $\G^c_K\ra \G^c_L$. 
This motivated the development of the {\em almost abelian anabelian geometry}.

%The correspondence works best for function fields over $\bar{\F}_p$ or
%some other algebraically closed fields of finite characteristics 
%(see Section~\ref{sect:almost}).

\

We now describe a recent related result in Galois cohomology, 
which could be considered as one of the incarnations of the general principle
formulated above.  
Let $G$ be a group and $\ell$ a prime number. 
The descending $\ell^n$-central series of $G$ is given by
$$
G^{(1,n)}=G,\quad G^{(i+1,n)}:=(G^{(i,n)})^{\ell^n} [  G^{(i,n)},G], \quad i=1, \ldots. 
$$
We write
$$
G^{c,n}=G/G^{(3,n)}, \quad G^{a,n}=G/G^{(2,n)}, 
$$
so that
$$
G^c=G^{c,0}, \quad G^a=G^{a,0}.
$$

\begin{theo}[Chebolu--Efrat--Min{\'a}{\v{c}} \cite{efrat-minac}]
Let $K$ and $L$ be fields containing $\ell^n$-th roots of 1 and 
$$
\Psi\colon \G_K\ra \G_L
$$
a continuous homomorphism. The following are equivalent:
\begin{itemize}
\item[(i)] the induced homomorphism
$$
\Psi^c \colon \G_K^{c,n} \ra \G_L^{c,n}
$$
is an isomorphism; 
\item[(ii)] 
the induced homomorphism
$$
\Psi^*\colon \rH^*(\G_L, \Z/\ell^n)\ra  \rH^*(\G_K, \Z/\ell^n)
$$
is an isomorphism.
\end{itemize}
\end{theo}

%%%%%%%%%%%%%%%%%%%%%%%%%%%%

\section{Commuting pairs and valuations}
\label{sect:valuations}

A \emph{value group}, $\Gamma$, is a totally ordered (torsion-free) abelian group. 
A (nonarchimedean) \emph{valuation} on a field $K$
is a pair $\nu=(\nu,\Gamma_{\nu})$ consisting of a 
value group $\Gamma_{\nu}$ and a map
$$
\nu\,:\, K\ra \Gamma_{\nu,\infty} = \Gamma_{\nu}\cup \infty
$$
such that
\begin{itemize}
\item $\nu\,:\, K^{\times}\ra \Gamma_{\nu}$ is a surjective homomorphism;
\item $\nu(\kappa+\kappa')\ge 
\min(\nu(\kappa),\nu(\kappa'))$ for all $\kappa,\kappa'\in K$;
\item $\nu(0)=\infty$.
\end{itemize}
The set of all valuations of $K$ is denoted by $\Val_K$.

Note that $\bar{\mathbb F}_p$ admits only the trivial valuation; 
we will be mostly interested in function fields $K=k(X)$ over $k=\bar{\mathbb F}_p$. 
A valuation is a \emph{flag map} on $K$: every finite-dimensional $\bar{\F}_p$-subspace, 
and also $\mathbb F_p$-subspace, 
$V\subset K$ has a flag $V=V_1\supset V_2 \ldots$ such that $\nu$ is constant on 
$V_j\setminus V_{j+1}$. 
Conversely, every flag map gives rise to a valuation. 

Let $K_{\nu}$, $\mathfrak o_{\nu}, \mathfrak m_{\nu}$, and 
$\KK_{\nu}:=\mathfrak o_{\nu}/\mathfrak m_{\nu}$  be
the completion of $K$ with respect to $\nu$, the valuation ring of $\nu$,
the maximal ideal of $\mathfrak o_{\nu}$, and the residue field, respectively.
A valuation of $K=\bar{\F}_p(X)$,  is called {\em divisorial} if the residue field
is the function field of a divisor on $X$; the set of such valuations is denoted by $\DVal_K$.  
We have exact sequences: 
\begin{equation*}
\label{eqn:1}
1\ra \mathfrak o_{\nu}^{\times}\ra K^{\times}\rightarrow \Gamma_{\nu}\ra 1
\end{equation*}
\begin{equation*}
\label{eqn:2}
1\ra (1+\mathfrak m_{\nu})\ra \mathfrak o_{\nu}^{\times}\ra \KK_{\nu}^{\times}\ra 1.
\end{equation*}
A homomorphism $\chi : \Ga_{\nu} \to \mathbb Z_{\ell}(1)$ gives rise to
a homomorphism
$$
\chi \circ \nu \, :\, K^{\times} \to \Z_{\ell}(1),
$$
thus to an element of $\G^a_{K}$, 
an \emph{inertia element} of $\nu$. 
These form the \emph{inertia} subgroup $\I^a_{\nu}\subset \G^a_{K}$. 
The \emph{decomposition group} $\D^a_{\nu}$
is the image of $\G^a_{K_{\nu}}$ in $\G^a_{K}$. 
We have an embedding 
$\G^a_{K_{\nu}}\hookrightarrow \G^a_K$ and an isomorphism 
$$
\D^a_{\nu}/\I^a_{\nu}\simeq \G^a_{\KK_{\nu}}.
$$

We have a dictionary (for $K=k(X)$ and $k=\bar{\F}_p$):
$$
\begin{array}{rcl}
\G^a_K     & = & \{\text{homomorphisms } \gamma \,:\, K^{\times}/k^{\times} \ra \Z_{\ell}(1)\}, \\
\D^a_{\nu} & = & 
\{ \mu\in \G^a_K\,|\, \mu\,\,\,{\rm trivial }\,\, \, {\rm on}\,\,\,
(1+\mathfrak m_{\nu})\},\\
\I^a_{\nu}& = & 
\{ \iota\in \G^a_K\,|\, \iota\,\,\,{\rm trivial }\,\, \, {\rm on}\,\,\,
\mathfrak o_{\nu}^{\times}\}.
\end{array}
$$
In this language, inertia elements define flag maps on $K$. 
If $E\subset K$ is a subfield, the 
corresponding homomorphism of Galois groups $\G_K\ra \G_E$ 
is simply the restriction of special $\Z_{\ell}(1)$-valued functions 
on the space $\P_k(K)$ to the projective subspace $\P_k(E)$. 

\

\noindent
The following result is fundamental in our approach to anabelian geometry. 

\begin{theo} \cite{BT}, \cite[Section 4]{bt0}
\label{thm:sigma}
Let $K$ be any field containing a subfield $k$ with $\# k\ge 11$. 
Assume that there exist nonproportional homomorphisms
$$
\gamma,\gamma' \colon K^{\times}\ra R
$$
where $R$ is either $\Z$, $\Z_{\ell}$ or $\Z/\ell$,
such that
\begin{itemize}
\item[(1)] $\gamma,\gamma'$ are trivial on $k^{\times}$;
\item[(2)] the restrictions of the 
$R$-module $\langle \gamma,\gamma', 1\rangle$ to every projective line 
$\P^1\subset \P_k(K) = K^{\times}/k^{\times}$ has $R$-rank $\le 2$.  
\end{itemize}
Then there exists a valuation $\nu$ of $K$ with value group $\Gamma_{\nu}$, 
a homomorphism $\iota \colon \Gamma_\nu\ra R$,
and an element $\iota_{\nu}$ in the $R$-span of $\gamma,\gamma'$
such that
$$
\iota_{\nu} =  \iota \circ \nu.
$$
\end{theo}

In (2), $\gamma, \gamma'$, and $1$ are viewed as functions on a projective line
and the condition states simply that these functions are linearly dependent.

This general theorem can be applied in the 
following contexts: $K$ is a function field over $k$, 
where $k$ contains all $\ell$-th roots of its elements and $R=\Z/\ell$, 
or $k=\bar{\F}_p$ with $\ell\neq p$ and $R=\Z_{\ell}$. 
In these situations, a homomorphism $\gamma \colon K^{\times}\ra R$ 
(satisfying the first condition) corresponds via
Kummer theory to an element in $\G^a_K/\ell$, resp. $\G^a_K$. 
Nonproportional elements $\gamma,\gamma'\in \G^a_K$ lifting to commuting elements in $\G^c_K$
satisfy condition (2). 
Indeed, for 1-dimensional function fields $E\subset K$
the group  $\G_{E}^c$ is a free central extension of $\G^a_E$. 
This holds in particular for $k(x)\subset K$. Hence $\gamma,\gamma'$ are proportional
on any $\P^1$ containing $1$; the restriction of $\sigma=\langle \gamma, \gamma'\rangle$ 
to such $\P^1$ is isomorphic to $\Z_\ell$. 
Property (2) follows since every $\P^1\subset P_k(K)$ 
is a translate, with respect to multiplication in $\P_k(K)=\K^{\times}/k^{\times}$, 
of the ``standard'' $\P^1=\P_k(k\oplus k x)$, $x\in K^{\times}$. 
Finally, the element $\iota_{\nu}$ obtained in the 
theorem is an inertia element 
for $\nu$, by the dictionary above.

\begin{coro}
\label{coro:ss}
Let $K$ be a  function field of an algebraic variety $X$ over 
an algebraically closed field $k$ of dimension $n$.
Let $\sigma\in \Sigma_K$ be a liftable subgroup. Then 
\begin{itemize}
\item $\rk_{\Z_{\ell}}(\sigma)\le n$;
\item there exists a valuation $\nu\in \Val_K$ and a subgroup $\sigma'\subseteq \sigma$ 
such that $\sigma'\subseteq \I^a_{\nu}$, $\sigma\subset \D^a_{\nu}$, and $\sigma/\sigma'$ is
topologically cyclic.
\end{itemize}
\end{coro}

Theorem~\ref{thm:sigma} and its
Corollary~\ref{coro:ss} allow to {\em recover} 
inertia and decomposition groups of valuations from 
$(\G^a_K,\Sigma_K)$. In reconstructions of function fields we 
need only divisorial valuations; these can be
characterized as follows: 

\begin{coro}
Let $K$ be a  function field of an algebraic variety $X$ over $k=\bar{\F}_p$ of dimension $n$.
If $\sigma_1,\sigma_2\subset \G^a_K$ are maximal liftable subgroups of 
$\Z_{\ell}$-rank $n$
such that $\I^a:=\sigma_1\cap \sigma_2$ is topologically cyclic 
then there exists a divisorial valuation
$\nu\in \DVal_K$ such that $\I^a=\I^a_{\nu}$. 
\end{coro}

Here we restricted to $k=\bar{\F}_p$ to avoid a discussion of mixed characteristic phenomena.
For example, the obtained valuation may be a divisorial valuation of a {\em reduction} 
of the field, and not of the field itself.

\

This implies that an isomorphism of Galois groups 
$$
\Psi \colon \G^a_K\ra \G^a_L
$$
inducing a bijection of the sets of liftable subgroups
$$
\Sigma_K = \Sigma_L
$$
induces a bijection of the sets of inertial and decomposition subgroups of valuations
$$
\{ \I^a_{\nu}\}_{\nu\in \DVal_K}  = \{\I^a_{\nu}\}_{\nu\in \DVal_L}, \quad 
\{ \D^a_{\nu}\}_{\nu\in \DVal_K}  = \{\D^a_{\nu}\}_{\nu\in \DVal_L}.
$$
Moreover, $\Psi$ maps topological generators $\delta_{\nu,K}$ of procyclic subgroups $\I^a_{\nu}\subset \G^a_K$, 
for $\nu\in \DVal_K$, to generators $\delta_{\nu,L}$ of corresponding inertia subgroups in $\G^a_L$, 
which pins down a generator up to the action of $\Z_{\ell}^{\times}$.

\

\

Here are two related results concerning the reconstruction of valuations.  

\begin{theo}[Efrat \cite{efrat}]
Assume that $\char(K)\neq \ell$,  $-1\in (K^{\times})^{\ell}$, and that
$$
\wedge^2(K^{\times}/(K^{\times})^{\ell}) \stackrel{\sim}{\lra}\rK^M_2(K)/\ell.
$$
Then there exists a valuation $\nu$ on $K$ such that 
\begin{itemize}
\item $\char(\KK_{\nu})\neq \ell$;
\item $\dim_{\F_{\ell}}(\Gamma_{\nu}/\ell) \ge \dim_{\F_{\ell}}(K^{\times}/(K^{\times})^\ell) -1$;
\item either  $\dim_{\F_{\ell}}(\Gamma_{\nu}/\ell) = \dim_{\F_{\ell}}(K^{\times}/(K^{\times})^\ell)$
or $\KK_{\nu}\neq \KK_{\nu}^{\ell}$. 
\end{itemize}
\end{theo}

In our terminology, under the assumption that $K$ contains an algebraically closed subfield $k$ 
and $\ell\neq 2$, the conditions mean that $G_K^a$ modulo $\ell$ is liftable, i.e., 
$G_K^c=G_K^a$. Thus there exists a valuation with  abelianized
inertia subgroup  (modulo $\ell$) of corank at most one, by Corollary~\ref{coro:ss}.  
The third assumption distinguishes the two cases, when the corank is zero versus one. 
In the latter case, the residue field $\KK_{\nu}$ has nontrivial $\ell$-extensions, 
hence satisfies $\KK_{\nu}^{\times}\neq (\KK_{\nu}^{\times})^{\ell}$.

\begin{theo}[Engler--K\"onigsmann \cite{engler-koe}/Engler--Nogueira, 
$\ell=2$ \cite{engler-nogu}]
Let $K$ be a field of characteristic $\neq \ell$ 
containing the roots of unity of order $\ell$.  
Then $K$ admits an $\ell$-Henselian 
valuation $\nu$ (i.e., $\nu$ extends uniquely
to the maximal Galois $\ell$-extension of $K$) 
with $\char(\KK_{\nu})\neq \ell$ 
and non-$\ell$-divisible $\Gamma_{\nu}$ if and only if
$\G_K$ is noncyclic and contains a nontrivial normal abelian subgroup. 
\end{theo}

Again, under the assumption that $K$ contains an algebraically closed field $k$, of characteristic $\neq \ell$, 
we can directly relate this result to our Theorem~\ref{thm:sigma} and 
Corollary~\ref{coro:ss} as follows:
The presence of an abelian normal subgroup in $\G_K$ means that modulo $\ell^n$ 
there is a nontrivial center. Thus there is a valuation $\nu$ such that 
$\G_K = \D_{\nu}$, the corresponding decomposition group. 
Note that the inertia subgroup $\I_{\nu}\subset \G_K$ maps injectively into $\I^a_{\nu}$.

\

We now sketch the proof of Theorem~\ref{thm:sigma}. 
Reformulating the claim, we see that the goal is to produce a 
{\em flag map} on $\P_k(K)$. Such a map $\iota$ 
jumps only on projective subspaces of $\P_k(K)$, i.e., 
every finite dimensional projective space $\P^n\subset \P_k(K)$ 
should admit a flag by projective subspaces
$$
\P^n\supset \P^{n-1}\supset ...
$$
such that $\iota$ is constant on $\P^r(k)\setminus \P^{r-1}(k)$, for all $r$. 
Indeed, a flag map defines a partial order on $K^{\times}$ which is preserved under shifts by
multiplication in $K^{\times}/k^{\times}$, hence a scale of $k$-subspaces
parametrized by some ordered abelian group $\Gamma$.

\

We proceed by contradiction. Assuming that the $R$-span 
$\sigma:=\langle \gamma,\gamma'\rangle$ does not contain 
a flag map we find a distinguished $\P^2\subset\P_k(K)$ such that 
$\sigma$ contains no maps which would be flag maps on this $\P^2$
(this uses that $\#k\ge 11$).  
To simplify the exposition, assume now that $k=\F_p$.

\

{\em Step 1.}
If $p>3$ then 
$\alpha \,:\, \mathbb P^2(\F_p)\ra R$ is a flag map
iff the restriction to \emph{every} $\mathbb P^1(\mathbb F_p)\subset \mathbb P^2(\mathbb F_p)$ is
a flag map, i.e., constant on the complement of one point.    

\

A counterexample for $p=2$ and $R=\Z/2$ is provided by the Fano plane:

\centerline{\includegraphics[width=.5\textwidth]{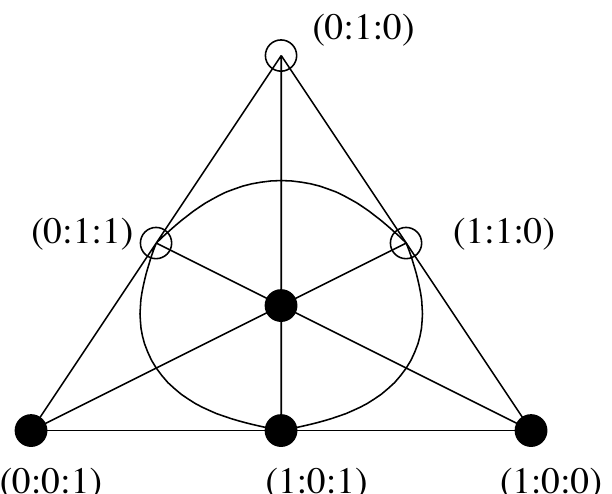}}

\

{\em Step 2.}
On the other hand, assumptions (1) and (2) imply that the map
$$
\begin{array}{ccc}
K^{\times}/k^{\times} = \mathbb P_k(K) &  \stackrel{\varphi}{\lra} &  \mathbb A^2(R) \\
      f & \mapsto &   (\gamma(f), \gamma'(f))
\end{array}
$$
maps every \emph{projective line} into an \emph{affine} line, a collineation. 
This imposes strong conditions on $\varphi=\varphi_{\gamma,\gamma'}$ and both $\gamma, \gamma'$. 
For example, for all $\P^2\subset \P_k(K)$ the image $\varphi(\P^2)$  
is contained in a union of an affine line and at most 
one extra point in $\A^2(R)$.

\

{\em Step 3.}
At this stage we are working with maps
$$
\P^2(\F_p) \ra \A^2(R), 
$$
preserving the geometries as above. Using Step 2 we may even reduce to considerations
of maps with image consisting of 3 points:
$$
\P^2(\F_p)\ra \{ \bullet, \circ, \star\} 
$$
and such that every line $\P^1(\F_p)\subset \P^2(\F_p)$ is mapped to exactly two points. 
Projective/affine geometry considerations produce a \emph{flag map} in 
the $R$-linear span of $\gamma,\gamma'$, contradicting the assumption.

\

The case of $\char(K)=0$ is more complicated (see \cite{BT}).

\

%%%%%%%%%%%%%%%%%%%%%%%%END of story%%%%%%%%%%%%%%%%%%%%5

\section{Pro-$\ell$-geometry}
\label{sect:pro-ell}

One of the main advantages in working with function fields $K$ 
as opposed to arbitrary fields is the existence of normal {\em models}, i.e., algebraic varieties $X$ with 
$K=k(X)$, and a {\em divisor} theory on $X$. 
Divisors on these models give rise to 
a rich supply of valuations of $K$, and we can employ geometric considerations  
in the study of relations between them.

We now assume that $k=\bar{\F}_p$, with $p\neq \ell$. 
Let $\Div(X)$ be the group of (locally principal) Weil divisors of $X$ and $\Pic(X)$ the 
Picard group. The exact sequence 
\begin{equation}
\label{eqn:seqq}
0\ra K^{\times}/k^{\times}\stackrel{\dv_X}{\lra} \Div(X)\stackrel{\pic}{\lra} \Pic(X)\ra 0, 
\end{equation}
allows us to connect functions $f\in K^{\times}$ to divisorial valuations, realized by 
irreducible divisors on $X$. 

We need to work simultaneously with two functors on $\Z$-modules of possibly infinite rank:
$$
M\mapsto M_{\ell}:=M\otimes \Z_{\ell} \quad \text{ and } M\mapsto \hat{M} := \lim_{\longleftarrow} M\otimes \Z/\ell^n. 
$$ 
Some difficulties arise from the fact that these are ``the same'' at each finite level, $\pmod{\ell^n}$.
We now recall these issues 
for functions, divisors, and Picard groups of normal projective models
of function fields (see \cite[Section 11]{bt0} for more details).  

Equation~\eqref{eqn:seqq} gives rise to an exact sequence
\begin{equation}
\label{eqn:seqq-times}
0\ra K^{\times}/k^{\times}\otimes \Z_{\ell}\stackrel{\dv_{X}}{\lra} \Div^0(X)_{\ell}\stackrel{\pic_{\ell}}{\lra} 
\Pic^0(X)\{\ell\}\ra 0. 
\end{equation}
where 
$$
\Pic^0(X)\{\ell\}=\Pic^0(X)\otimes \Z_{\ell}
$$ 
is the $\ell$-primary component of the torsion group of $k=\bar{\F}_p$-points of 
$\Pic^0(X)$, the algebraic group parametrizing 
classes of algebraically equivalent divisors modulo
rational equivalence.   
Put
$$
\cT_{\ell}(X):=\lim_{\longleftarrow}{\rm Tor}_1(\Z/\ell^n,\Pic^0(X)\{\ell\}).
$$ 
We have $\cT_{\ell}(X)\simeq \Z_{\ell}^{2\mathsf g}$, 
where $\mathsf g$ is the dimension of $\Pic^0(X)$.
In fact, $\cT_{\ell}$ is a contravariant functor, 
which stabilizes on some normal projective model $X$, 
i.e., $\cT_{\ell}(\tilde{X})=\cT_{\ell}(X)$ for all $\tilde{X}$ surjecting onto $X$. 
In the sequel, we will implicitly work with such $X$ and we write $\cT_{\ell}(K)$.

%The following lemma generalizes \cite[Lemmas 11.12 and 11.14]{bt0} 
%to normal varieties; it is essential in our reduction of Theorem~\ref{thm:0} 
%to the case of surfaces in \cite{bt1}.
%\begin{lemm}
%\label{lemm:tl}
%Let $K=k(X)$ be the function field of a normal
%projective variety $X\subset \P^N$ of dimension $\ge 3$. 
%For every divisorial valuation $\nu\in \DVal_K$ there is a canonical
%homomorphism:
%$$
%\xi_{\nu,\ell} \,:\, \cT_{\ell}(K)\ra \cT_{\ell}(\KK_{\nu}).
%$$
%Assume that $\nu$ corresponds to an irreducible normal 
%hyperplane section of $X$. 
%Then $\xi_{\nu,\ell}$ is an isomorphism.  
%\end{lemm}

Passing to pro-$\ell$-completions in \eqref{eqn:seqq-times} we obtain an exact sequence:
\begin{equation}
\label{eqn:seqq-pro-ell}
0\ra \cT_{\ell}(K)\ra \hat{K}^{\times}\stackrel{\dv_X}{\lra} \widehat{\Div^0}(X) \lra 0, 
\end{equation}
since $\Pic^0(X)$ is an $\ell$-divisible group. 
Note that all groups in this sequence are torsion-free. 
We have a diagram  

\centerline{
\xymatrix{ 
        & 0 \ar[r] & K^{\times}/k^{\times}\otimes\Z_{\ell}\ar[d]\ar[r]^{\dv_{X}} &\Div^0(X)_{\ell}\ar[d]\ar[r]^{\!\!\!\pic_{\ell}} & 
\Pic^0(X)\{\ell\} \ar[d] \ar[r] &  0  \\ 
0 \ar[r] &  \cT_{\ell}(K) \ar[r]  & \hat{K}^{\times} \ar[r]^{\dv_{X}} & \widehat{\Div^0}(X) \ar[r] & 0
}
}

\
Galois theory allows to ``reconstruct'' the second row of this diagram. The reconstruction of fields
requires the first row. The passage from the second to the first employs the theory of valuations.  
Every $\nu\in \DVal_K$ gives rise to a
homomorphism
$$
\nu\,:\, \hat{K}^{\times}\ra \Z_{\ell}.
$$
On a normal model $X$, where $\nu=\nu_D$ for some divisor $D\subset X$, 
$\nu(\hat{f})$ is the $\ell$-adic coefficient at $D$ of $\dv_X(\hat{f})$.
``Functions'', i.e., elements $f\in K^{\times}$, have {\em finite support} on 
models $X$ of $K$, i.e., only finitely many coefficients $\nu(f)$ are nonzero. 
However, the passage to blowups of $X$ introduces more and more divisors
(divisorial valuations) in the support of $f$. 
The strategy in \cite{bt0}, specific to dimension two,  
was to extract elements of $\hat{K}^{\times}$ 
with {\em intrinsically} finite support, using the interplay between
one-dimensional subfields $E\subset K$, i.e., projections of $X$ onto curves, 
and divisors of $X$, i.e., curves $C\subset X$.   
For example, Galois theory allows to distinguish valuations $\nu$ 
corresponding to {\em rational} and {\em nonrational} curves on $X$. 
If $X$ had only {\em finitely many} rational curves, 
then every blowup $\tilde{X}\ra X$ would have the same property. 
Thus elements $\hat{f}\in \hat{K}^{\times}$
with finite {\em nonrational} support, i.e., $\nu(f)=0$ for all but finitely many nonrational
$\nu$, have necessarily finite support on every model $X$ of $K$, 
and thus have a chance of being functions. A different geometric argument applies
when $X$ admits a fibration over a curve of genus $\ge 1$,  with rational generic fiber. 
The most difficult case to treat, surprisingly, is the case of rational surfaces. 
See Section 12 of \cite{bt0} for more details.

\

The proof of Theorem~\ref{thm:0} 
in \cite{bt1} reduces to dimension two, via Lefschetz pencils.

\section{Pro-$\ell$-$\rK$-theory}
\label{sect:wishful}

Let $k$ be an algebraically closed field of characteristic $\neq \ell$ 
and $X$ a smooth projective variety over $k$, with function field $K=k(X)$.
A natural generalization of \eqref{eqn:seqq} is the Gersten sequence 
(see, e.g., \cite{suslin-ob}): 
$$
0\ra \rK_2(X)\ra \rK_2(K)\ra \bigoplus_{x\in X_1} \rK_1(k(x))\ra \bigoplus_{x\in X_2} \Z \ra \CH^2(X) \ra 0, 
$$
where $X_d$ is the set of points of $X$ of codimension $d$ and $\CH^2(X)$ is the second Chow group of $X$.  
Applying the functor 
$$
M\mapsto M^{\vee}:=\Hom(M,\Z_{\ell})
$$ 
and using the duality
$$
\G^a_K = \Hom(K^{\times},\Z_{\ell}) 
$$
we obtain a sequence

\

\centerline{
\xymatrix{
\rK_2(X)^{\vee} & \rK_2(K)^{\vee} \ar[l] & \ar[l] \prod_{D\subset X} \G^a_{k(D)} 
 }  
}

\

\noindent
Dualizing the sequence
$$
0\ra I_K \ra \wedge^2(K^{\times}) \ra \rK_2(K)\ra 0
$$
we obtain
$$
I_K^{\vee} \leftarrow \wedge^2(\G_K^a)  \leftarrow    \rK_2(K)^{\vee}  \leftarrow 0 
$$
On the other hand, we have the following exact sequences:
$$
0\ra Z_K\ra \G^c_K\ra \G^a_K \ra 0
$$
and the resolution of $Z_K=\left[\G^c_K,\G^c_K\right]$ 
$$
0\ra \rR(K)\ra \wedge^2(\G^a_K) \ra Z_K\ra 0.
$$

Recall that  $\G^a_K = \Hom(K^\times/k^\times,\Z_\ell)$ 
is a torsion-free $\Z_\ell$-module, with topology induced from 
the discrete topology on 
$K^\times/k^\times$. 
Thus any primitive finitely-generated subgroup 
$A\subset K^\times/k^\times$ is a direct summand and 
defines a continuous surjection 
$\G^a_K\ra \Hom(A,\Z_\ell)$.
The above topology on $\G^a_K$ defines a natural topology on 
$\wedge^2(\G^a_K)$.
On the other hand, we have  a topological profinite group
$\G^c_K$ with topology induced by finite  $\ell$-extensions of $K$, 
which contains a closed abelian subgroup $Z_K= [\G^c_K,\G^c_K]$.

\begin{prop} \cite{B-1}
\label{prop:rr}
We have
$$
\rR(K) = (\Hom(\rK_2(K)/{\rm Image}(k^\times\otimes K^\times), \Z_\ell)=\rK_2(K)^{\vee}.
$$
\end{prop}

\begin{proof}
There is {\em continuous} surjective homomorphism 
$$
\begin{array}{rcl}
\wedge^2(\G^a_K) & \to  & Z_K \\
\gamma\wedge \gamma' & \mapsto & [\gamma,\gamma']
\end{array}
$$
The kernel $\rR(K)$ is a profinite group with the induced topology.
Any $r\in \rR(K)$ is trivial on symbols $(x,1-x) \in \wedge^2(K^\times/k^{\times})$ 
(since the corresponding elements are trivial in $\rH^2(\G^a_K,\Z/\ell^n)$, for all $n\in \N$). 
Thus  $\rR(K)\subseteq \rK_2(K)^{\vee}$.

Conversely, let $\alpha \in \rK_2(K)^{\vee}\setminus \rR(K)$; so that it projects nontrivially to 
$Z_K$, i.e., to a nontrivial 
element modulo $\ell^n$, for some $n\in \N$. 
Finite quotient groups of $\G^c_K$ 
with $Z(G_i^c)= [G_i^c, G_i^c]$ form a basis of topology on $\G^c_K$.  
The induced surjective homomorphisms $\G^a_K\to G_i^{a}$ define
surjections $\wedge^2(\G^a_K)\to [G_i,G_i]$ and 
$$
\rR(K)\to \rR_i:=\Ker(\wedge^2(G_i^{a})\to [G_i,G_i]).
$$
Fix a $G_i$ such that $\alpha$ is nontrivial of $G_i^c$.
Then the element $\alpha$ is nonzero in the image of 
$\rH^2(G_i^a,\Z/\ell^n) \ra \rH^2(G_i^c,\Z/\ell^n)$.
But this is incompatible with 
relations in $\rK_2(K)$, modulo $\ell^n$.   
\end{proof}

It follows that $\rR(K)$ contains a 
distinguished $\Z_{\ell}$-submodule
\begin{equation}
\label{eqn:wedge}
\rR_{\wedge}(K) = \text{ Image of }   \prod_{D\subset X} \G^a_{k(D)}
\end{equation}
and that 
$$
\rK_{2}(X)^{\vee}\supseteq \rR(K)/\rR_{\wedge}(K).
$$
In general, let 
$$
\rK_{2,nr}(K) = \Ker( \rK_2(K)\ra \bigoplus_{\nu\in \DVal_K} \KK_{\nu}^{\times})
$$  
be the {\em unramified} $\rK_2$-group.  
Combining Proposition~\ref{prop:rr} and 
\eqref{eqn:wedge}, we find that
$$
\widehat{\rK_{2,nr}}(K) \subseteq \Hom( \rR(K)/\rR_{\wedge}(K), \Z_{\ell}). 
$$
This sheds light on the connection between relations in $\G^c_K$
and the $\rK$-theory of the field, more precisely, the unramified Brauer group of $K$. 
This in turn helps to reconstruct multiplicative groups of 1-dimensional subfields of $K$. 

\

We now sketch a closely related, alternative strategy for the reconstruction of these subgroups
of $\hat{K}^{\times}$ from Galois-theoretic data. 
We have a diagram

\

\centerline{
\xymatrix{ 
0\ar[r] & \G_K^c \ar[d] \ar[r]  & \prod_E \G_K^c \ar[d] \ar[r]^{\rho^c_E}  & \G^c_E  \ar[d]\\
0\ar[r] & \G_K^a        \ar[r]  & \prod_E \G_K^a         \ar[r]^{\rho^a_E} & \G^a_E  
}
}

\

\noindent
where the product  is taken over all normally closed 1-dimensional subfields $E\subset K$, 
equipped with the direct product topology, and 
the horizontal maps are closed embeddings.
Note that $\G^a_K$ is a primitive subgroup
given by equations 
$$
\G^a_K = \{ \gamma \,\, \mid  \,\, (xy)(\gamma) - (x)(\gamma)- (y)(\gamma) = 0 \}  \subset \prod_E \G^a_{E}
$$
where  $x,y$ are algebraically independent in $K$
and   $xy,x,y\in K^{\times}$ are considered as functionals
on $\G^a_{k(xy)}, \G^a_{k(x)}, \G^a_{k(y)}$, respectively.
The central subgroup 
$$
Z_K\subset \G^c_K\subset \prod_E \wedge^2(\G^a_{E})
$$
is the image of $\wedge^2(\G^a_K)$ in $\prod_E \wedge^2(\G^a_{E})$.
Thus for any finite quotient $\ell$-group $G$ of $\G^c_K$ there
is an intermediate quotient which is a subgroup of finite
index in the product of free central extensions.
The following fundamental conjecture lies at the core
of our approach.

\begin{conj}
\label{conj:hope}
Let $K$ be a function field over $\bar{\F}_p$, with $p\neq \ell$,
$F^a$ a torsion-free topological $\Z_\ell$-module of
infinite rank. 
Assume that 
$$
\Psi_F^a : \G^a_K\to F^a
$$ 
is a continuous surjective 
homomorphism such that 
$$
\rk_{\Z_{\ell}}(\Psi_F^a(\sigma))\le 1
$$
for all liftable subgroups $\sigma \in \Sigma_K$. 
Then there exist a 1-dimensional subfield $E\subset K$, a subgroup $\tilde{F}^a\subset F^a$
of finite corank, and a diagram

\[
\centerline{
\xymatrix@!{
      & \G^a_K\ar[dl] \ar[dr] & \\
 \G^a_{E} \ar[rr]  &   &           \tilde{F}^a \\
}
}
\]
\end{conj}

We expect that $\tilde{F}_a=F_a$, when $\pi_1(X)$ is finite.   
Note that there can exist at most one normally closed subfield 
$E\subset F$ satisfying this property. 

\

The intuition behind this conjecture is that such maps should arise from 
surjective homomorphisms onto free central extensions, i.e., 
we should be able to factor as follows:
$$
\Psi_F^c= \G_K^c\stackrel{\rho^c_F}{\lra} \G^c_F\ra F^c
$$
where $F^c$ is a free central extension of $F^a$:
$$
0\ra \wedge^2(F^a)\ra F^c\ra F^a\ra 0.
$$
We can prove the conjecture under some additional geometric assumptions.
Assuming the conjecture, the proofs in \cite{bt0}, \cite{bt1} would
become much more straightforward.
Indeed, consider the diagram

\

\centerline{
\xymatrix{ 
\G^a_K \ar[r]^{\sim} &  \G^a_L \ar[d]  \\
              &  \G^a_F
}
}

\

\noindent
Applying Conjecture~\ref{conj:hope} we find a unique
normally closed subfield $E\subset K$ and  a canonical isomorphism
$$
\Psi\colon \G^a_E\ra \G^a_F, \quad F\subset L, 
$$
Moreover, this map gives a bijection between the set of 
inertia subgroups of divisorial valuations on $E$ and of $F$;
these are the images of inertia subgroups of divisorial valuations on $K$ and $L$.  
At this stage, the simple rationality argument (see \cite[Proposition 13.1 and Corollary 15.6]{bt0})  
implies that 
$$
\Psi^* \colon \hat{L}^{\times} \stackrel{\sim}{\lra} \hat{K}^{\times}
$$
induces an isomorphism
$$
L^{\times} /l^{\times} \otimes \Z_{(\ell)} \stackrel{\sim}{\lra}  
\epsilon \left( K^{\times} /k^{\times} \otimes \Z_{(\ell)}\right),
$$
for some $\epsilon \in \Z_{\ell}^{\times}$, 
respecting multiplicative subgroups of 1-dimensional subfields. Moreover,
for each 1-dimensional rational subfield $l(y)\subset L$ we obtain
$$
\Psi^*(l(y)^{\times}/l^{\times})=  \epsilon \cdot \epsilon_y \cdot \left(k(x)^{\times}/k^{\times}\right)
$$
for some $\epsilon_y\in \Q$. Proposition 2.13 in \cite{bt0} shows that 
this implies the existence of subfields  $\bar{L}$ and $\bar{K}$ such that 
$L/\bar{L}$ and $K/\bar{K}$ are purely inseparable extensions and  
such that $\epsilon^{-1}\cdot \Psi^*$ induces an {\em isomorphism} of multiplicative
groups 
$$
\P(\bar{L})=\bar{L}^\times /l^{\times} \stackrel{\sim}{\lra} \P(\bar{K})=\bar{K}^{\times} /k^{\times}.
$$
Moreover, this isomorphism maps lines $\P^1\subset \P(l(y))$ to lines $\P^1\subset \P(k(x))$. 
Arguments similar to those in Section~\ref{sect:k-theory} allow us to show that 
$\Psi^*$ induces an bijection of the sets of all {\em projective lines} of the projective structures. 
The Fundamental theorem of projective geometry (Theorem~\ref{thm:recon-fields}) 
allows to match the additive structures and leads
to an isomorphism of fields.

The proof of Theorem~\ref{thm:0} in \cite{bt0} is given for the case
of the fields of transcendence degree two.
However, the general case immediately follows by applying 
Theorem~\ref{thm:recon-fields}
from Section~\ref{sect:projective} (or \cite{bt1}).
Indeed, it suffices to show that for all $x,y\in L^\times/l^\times$
$$
\Psi^*(l(x,y)^\times/l^\times)\subset 
\overline{k(x,y)}^\times/k^\times \otimes \Z_{(\ell)}
\subset K^\times/k^{\times} \otimes \Z_{(\ell)}.
$$

Note that the groups $\overline{l(x)}^\times/l^\times$ map into
subgroups $\overline{k(x)}^\times/k^\times\times \Z_{(\ell)}$ 
since 
$\Psi^*$ satisfies the conditions of \cite[Lemma 26]{bt1}, i.e., 
the symbol 
$$
(\Psi^*(y),\Psi^*(z)) \in \rK^M_2(K)\otimes \Z_{\ell}
$$ 
is infinitely
$\ell$-divisible, for any $y,z \in  \overline{l(x)}^\times/l^\times$.  Thus 
$$
\Psi^*(\overline{l(x/y)}^\times) \in \overline{k(x,y)}^\times/k^\times \otimes \Z_{(\ell)}
$$
and similarly for
$\Psi^*(\overline{l(x+by)}^\times)/l^\times,   b\in l$, 
since by multiplicativity
$$
\Psi^*(\overline{l(x+y)}^\times/l^\times) 
\subset \cup_n (y^n \cdot \Psi^*(\overline{l(x/y+ b)}^\times/l^\times) =  
\cup_n (y^n \cdot \Psi^*(\overline{l(x/y)}^\times/l^\times)).
$$
Thus 
$$\Psi^*(x/y)/ l^{\times}, \Psi^*(x+y)/ l^\times\in
\overline{k(x,y)}^\times/k^\times \otimes \Z_{(\ell)},
$$ 
so that Theorem~\ref{thm:0}, for fields of arbitrary transcendence degree, follows from
the result for transcendence degree two.

\section{Group theory}
\label{sect:group}

Our intuition in Galois theory and Galois cohomology is based on 
the study of {\em finite} covers and {\em finite} groups. 
Our goal is to recover fields or some of their invariants 
from invariants of their absolute Galois groups and their
quotients.

In this section, we study some group-theoretic constructions 
which appear, in disguise, in the study of function fields. 
Let $G$ be a finite group. We have
$$
G^c=G/[[G,G],G], \quad G^a=G/[G,G].
$$
Let 
$$
\rB_0(G):=\Ker\left( \rH^2(G,\Q/\Z)\ra \prod_{B} \rH^2(B,\Q/\Z) \right)
$$
be the subgroup of those Schur multipliers which restrict trivially to all
bicyclic subgroups $B\subset G$.  
The first author conjectured in \cite{bog-stable} that 
$$
\rB_0(G)=0
$$
for all finite simple groups. 
Some special cases were proved in 
\cite{bog-unram}, and the general case was settled \cite{kun}.

In computations of this group it is useful to keep in mind the following diagram

\

\[ 
\centerline{
\xymatrix@C=15pt@R=20pt{
 \rB_0(G^c)\ar[d]                      & \rH^2(G^a)\ar@{>>}[d] 
%\ar[dr]^{\pi_a^*} 
& \rB_0(G) \ar[d] \\
 \rH^2(G^c)\ar[d] \ar@{=}[r]           & \rH^2(G^c) \ar[d]\ar[r]                & \rH^2(G) \ar[d]  \\
\prod_{B\subset G^c}\rH^2(B)\ar@{=}[r] &\prod_{B\subset G^c}\rH^2(B)\ar@{>>}[r] & \prod_{B\subset G} \rH^2(B). 
}
}
\] 

\

\noindent
Thus we have a homomorphism 
$$
\rB_0(G^c)\ra \rB_0(G). 
$$
We also have an isomorphism
$$
\Ker \left( \rH^2(G^a,\Q/\Z) \ra \rH^2(G,\Q/\Z)\right)= 
\Ker \left( \rH^2(G^a,\Q/\Z) \ra \rH^2(G^c,\Q/\Z)\right)
$$
Combining with the fact that
$\rB_0(G^c)$ is in the image of 
$$
\pi_a^* \colon \rH^2(G^a,\Q/\Z) \ra \rH^2(G,\Q/\Z)
$$
this implies that 
\begin{equation}
\label{eqan:bogg}
\rB_0(G^c)\hookrightarrow \rB_0(G). 
\end{equation}

Let $\ell$ be a prime number.
We write $G_{\ell}$ for the maximal $\ell$-quotient of $G$ and fix an $\ell$-Sylow subgroup
$\Syl_{\ell}(G)\subset G$, all considerations below are independent of the conjugacy class. 
We have a diagram

\

\centerline{
\xymatrix{
                         &   G       \ar@{>>}[d]  \ar@{>>}[r] & G^c\ar@{>>}[d]\ar@{>>}[r]  & G^a\ar@{>>}[d] \\
\Syl_\ell(G) \ar@{>>}[r] &   G_{\ell}\ar@{>>}[r]              & G^c_{\ell} \ar@{>>}[r]      & G^a_{\ell}   
}}

\

\noindent
Note that 
$$
G^c_{\ell}=\Syl_{\ell}(G^c), \quad \text{ and } \quad  G^a_{\ell}=\Syl_{\ell}(G^a),
$$
but that, in general, $\Syl_{\ell}(G)$ is much bigger than $G_{\ell}$.

\

We keep the same notation when working with pro-$\ell$-groups.

\

\begin{prop} 
\label{prop:sylly}
\cite{B-2}
Let $X$ be a projective algebraic variety of dimension $n$ over a field $k$. 
Assume that $X(k)$ contains a smooth point. Then 
$$
\Syl_{\ell}(G_{k(X)})=\Syl_{\ell}(G_{k(\P^n)}).
$$
\end{prop}

\begin{proof}
First of all, let $X$ and $Y$ be algebraic varieties over a field $k$ 
with function fields $K=k(X)$, resp. $L=k(Y)$. 
Let $X\ra Y$ be a map of degree $d$ and 
$\ell$ a prime not dividing $d$ and $\char(k)$.
Then 
$$
\Syl_{\ell}(G_K)=\Syl_{\ell}(G_L).
$$
Let $X\ra \P^{n+1}$ be a birational embedding as a (singular) 
hypersurface of degree $d'$. 
Consider two projections onto $\P^n$: the first, $\pi_x$ 
from a smooth point $x$ in the image of $X$
and the second, $\pi_y$,  from a point $y$ in the complement of $X$ in $\P^{n+1}$.  
We have $\deg(\pi_y)=d'$ and $\deg(\pi_y) - \deg(\pi_x)=1$, 
in particular, one of these degrees is coprime to $\ell$. 
The proposition follows from the first step. 
\end{proof}

\begin{rema}
This shows that the full Galois group $G_K$ is, in some sense, {\em too large}:
the isomorphism classes of its $\ell$-Sylow subgroups depend only on the dimension and the 
ground field. We may write
$$
\Syl_{\ell}(G_K)=\Syl_{\ell,n,k}.
$$
In particular, they {\em do not} determine the function field.
However, the maximal pro-$\ell$-quotients do \cite{mochizuki}, \cite{pop}. 
Thus we have a surjection from a {\em universal} group, 
depending only on the dimension and ground field $k$, onto a 
{\em highly individual} group $\G^c_K$, 
which by Theorem~\ref{thm:0} determines the field $K$, 
for $k=\bar{\F}_p$, $\ell\neq p$, and $n\ge 2$.

The argument shows in particular
that the  group $\Syl_{\ell,k,n}$ belongs to the class of self-similar
groups. Namely any open subgroup of finite index in $\Syl_{\ell,k,n}$ 
is isomorphic to $\Syl_{\ell,k,n}$. The above construction provides with
isomorphisms parametrized by smooth $k$-points of $n$-dimensional
algebraic varieties. Note that the absence of smooth $k$-points
in $K$ may lead to a nonisomorphic group $\Syl_{\ell,k,n}$, as seen already in 
the example
of a conic $C$ over $k=\mathbb R$ with $C(\mathbb R)=\emptyset$ \cite{B-2}.
\end{rema}

\

\begin{theo}\cite[Thm. 13.2]{B-1} 
\label{thm:b0g}
Let $G_K$ be the Galois group of a function field $K=k(X)$ over 
an algebraically closed ground field $k$. 
Then, for all $\ell\neq \char(k)$ we have
$$
\rB_{0,\ell}(G_K)= \rB_0(\G^c_K).
$$
\end{theo}

%{\bf explain that in $\trdeg\ge 2$ the groups $\rB_0(\G^c_K)$ are very large.} 

Here is a sample of known facts:
\begin{itemize}
\item if $X$ is stably rational over $k$, then  
$$
\rB_0(G_K)=0;
$$ 
\item if $X=V/G$, where $V$ is a faithful representation of $G$ over an algebraically closed
field of characteristic coprime to the order of $G$, and $K=k(X)$,  then 
$$
\rB_0(G)=\rB_0(G_K),
$$
thus nonzero in many cases.
\end{itemize}

Already this shows that the groups $G_K$ are quite special. 
The following ``Freeness conjecture'' is related to the 
Bloch--Kato conjecture discussed in Section~\ref{sect:bloch-kato}; it would imply that
all cohomology of $G_K$ is induced from metabelian finite $\ell$-groups.

\begin{conj}[Bogomolov]
For $K=k(X)$, with $k$ algebraically closed of characteristic $\neq \ell$, 
let 
$$
\Syl^{(2)}_{\ell,n,k}=[\Syl_{\ell,n,k},\Syl_{\ell,n,k}], 
$$   
and let $M$ be a finite  $\Syl^{(2)}_{\ell,n,k}$-module. 
Then
$$
\rH^i(\Syl^{(2)}_{\ell,n,k}, M)=0,  \quad \text{ for all } \quad i\ge 2.
$$
\end{conj}

Further discussions in this direction, in particular, concerning the connections between the Bloch--Kato conjecture, 
``Freeness'', and the Koszul property of the algebra $\rK_*^M(K)/\ell$, can be found in \cite{posi} and 
\cite{posi-vish}.

\section{Stabilization}
\label{sect:stabilization}

The varieties $V/G$ considered in the Introduction seem very special. On the other hand,
let $X$ be any variety over a field $k$ and let 
$$
G_{k(X)}\ra G
$$
be a continuous homomorphism from its Galois group onto some {\em finite} group. 
Let $V$ be a faithful representation of $G$. 
Then we have two homomorphisms (for cohomology with finite coefficients and trivial action)
$$
\kappa_X \colon \rH^*(G)\ra \rH^*(G_{k(X)}) 
$$
and 
$$
\kappa_{V/G}\colon \rH^*(G)\ra \rH^*(G_{k(V/G)}).
$$
These satisfy
\begin{itemize}
\item $\Ker(\kappa_{V/G})\subseteq \rH^*(G)$ is independent of $V$, and the quotient 
$$
\rH^*_{s}(G):=\rH^*(G)/\Ker(\kappa_{V/G})
$$
is well-defined;
\item $\Ker(\kappa_{V/G})\subseteq \Ker(\kappa_X)$. 
\end{itemize}
The groups $\rH^i_{s}(G)$ are called {\em stable} cohomology groups of $G$.
They were introduced and studied by the first author in \cite{bog-stable}.
{\em A priori}, these groups depend on the ground field $k$. 
We get a surjective homomorphism
$$
\rH^*_{s}(G)\ra \rH^*(G)/\Ker(\kappa_X). 
$$
This explains the interest in stable cohomology---all 
group-cohomological invariants arising from 
finite quotients of $G_{k(X)}$ arise from similar invariants of $V/G$. 
On the other hand, there is no effective procedure for the computation of
stable cohomology, except in special cases. 
For example, for abelian groups the stabilization
can be described already on the group level:

\begin{prop}[see, e.g., \cite{bog-stable}] 
\label{prop:stab-abel}
Let $G$ be a finite abelian group and
$\sigma :\Z^m\to G$ a surjective homomorphism.
Then $\kappa^* : \rH^*(G)\to \rH^*(\Z^m)$ coincides with
the stabilization map, i.e., 
$$
\Ker(\kappa^*)=\Ker(\kappa_{V/G})
$$ 
for any 
faithful representation $V$ of $G$, 
for arbitrary ground fields $k$ with
$\char(k)$ coprime to the order of $G$.
\end{prop}

Geometrically, stabilization is achieved on the variety $T/G\subset V/G$, where
$G$ acts faithfully on $V$ by diagonal matrices and $T\subset V$ 
is a $G$-invariant subtorus in $V$ (see, e.g., \cite{bog-linear}).

Similar actions exist for any finite group $G$:
there is faithful representation
$V$ and a torus $T\subset \Aut(V)$, with normalizer $N=N(T)$ 
such that $G\subset N\subset \Aut(V)$, and such
that $G$ acts {\em freely} on $T$.
We have an exact sequence
$$
1\ra \pi_1(T)\ra \pi_1(T/G)\ra G\ra 1
$$
of topological fundamental groups. Note that $\pi_1(T)$
decomposes as a sum of $G$-permutation modules and that 
$\pi_1(T/G)$ is torsion-free of 
cohomological dimension $\dim(T) = \dim(V)$.
Torus actions were considered by Saltman \cite{saltman-torus}, 
and the special case of actions coming
from restrictions to open tori in linear representations
by the first author in \cite{bog-linear}.

\

The following proposition, a consequence of the Bloch--Kato conjecture, 
describes a partial stabilization for central extensions of abelian groups.  

\begin{prop}
\label{prop:stab-partial}
Let $G^c$ be a finite $\ell$-group which is a central extension of an abelian group
\begin{equation}
\label{eqn:induced}
0\ra Z\ra G^c\ra G^a\ra 0, \quad Z=[G^c,G^c],
\end{equation}
and $K=k(V/G^c)$. 
Let 
$$
\phi_a\colon \Z^m_{\ell}\ra G^a
$$
be a surjection and
$$
0\ra Z\ra D^c\ra \Z_{\ell}^m \ra 0
$$
the central extension induced from  \eqref{eqn:induced}.
Then 
$$
\Ker(\rH^*(G^a)\to \rH^*(D^c))=
\Ker(\rH^*(G^a)\to \rH^*(\Gal_{K})), 
$$
for cohomology with $\Z/\ell^n$-coefficients, $n\in \N$. 
\end{prop}

\begin{proof}
Since $\G^a_K$ is a torsion-free $\Z_{\ell}$-module we have a diagram

\centerline{
\xymatrix{ 
 & G_K \ar@{>>}[r]    & \G_K^c\ar[r] \ar@{-->>}[d]   & \G_K^a \ar[r] \ar@{>>}[d]      & 0   \\
0\ar[r]& Z \ar[r]\ar@{=}[d] & D^c\ar[r]   \ar@{>>}[d]        & \Z_{\ell}^m \ar@{>>}[d]^{\phi_a} \ar[r] & 0\\
0\ar[r]& Z   \ar[r]           & G^c \ar[r]                & G^a \ar[r]           & 0
} 
}

\

\noindent
By Theorem~\ref{thm:bkk}, 
$$
\Ker\left(\rH^*(G^a)\to \rH^*(G_K) \right) 
= 
\Ker\left(\rH^*(G^a)\to \rH^*(\G^c_{K})\right).
$$
Note that
$$
I:=\Ker\left(\rH^*(G^a)\to \rH^*(D^c)\right)
$$ 
is an ideal generated by its degree-two elements $I_2$ and that 
$$
I_2=\Ker\left(\rH^2(G^a)\to \rH^2(G^c)\right) \oplus 
\delta(\rH^1(G^a)).
$$
Similarly, for all intermediate $D^c$
$$
\Ker\left(\rH^*(G^a)\to \rH^*(D^c)\right) 
$$ 
is also generated by $I_2$, and hence equals $I$.
\end{proof}

\begin{coro} 
Let $G^c$ be a finite $\ell$-group as above,  
$\rR\subseteq \wedge^2(G^a)$ the 
subgroup of relations defining $D^c$, and  let 
$$
\Sigma=\{ \sigma_i\subset G^a\}
$$ 
be the set of subgroups of $G^a$ liftable to abelian subgroups of $G^c$. 
Then the image of $\rH^*(G^a,\Z/\ell^n)$ in
$\rH^*_s(G^c,\Z/\ell^n)$ coincides with $\wedge^*(G^a)^*/ I_2$, where
$I_2\subseteq \wedge^2(G^a)$ are the elements orthogonal to $\rR$  
(with respect to the natural pairing).
\end{coro}

\begin{lemm} 
For any finite group $G^c$ there is a torsion-free group
$\G^c$ with $\G^a= \Z_\ell^n$ and $[\G^c,\G^c] = \Z_\ell^m$ with a natural
surjection $\G^c\to G^c$ and a natural embedding
$$
\Ker(\rH^2(G^a)\to \rH^2(G^c)) = 
\Ker(\rH^2(\G^a)\to \rH^2(\G^c)),
$$
for cohomology with $\Q_{\ell}/\Z_{\ell}$-coefficients.
\end{lemm}

\begin{proof}
Assume that we have a diagram of central extensions

\

\centerline{
\xymatrix{
0 \ar[r] & Z_{\G}\ar[r]\ar[d] & \G^c\ar[r]^{\pi_{a,\G}} \ar@{>>}[d]_{\pi_c}  & \G^a \ar[r] \ar@{=}[d] & 0  \\ 
0 \ar[r] & Z_{\cH}\ar[r] & \H^c\ar[r]_{\pi_{a,\cH}}                      & \cH^a \ar[r]                  & 0 
}
}

\

\noindent
with $\G^a = \cH^a$, $Z_{\G}$, and $Z_{\cH}$ finite rank torsion-free $\Z_{\ell}$-modules.
Assume that
$$
\Ker(\pi_{a,\cH}^*):= \Ker\left(\rH^2(\cH^a, \Z_{\ell})\to \rH^2(\cH^c, \Z_{\ell})\right)
$$
coincides with 
$$
\Ker(\pi_{a,\G}^*):= \Ker\left(\rH^2(\G^a, \Z_{\ell})\to \rH^2(\G^c, \Z_{\ell})\right).
$$
Then there is a section 
$$
s:  \cH^c\to \G^c, \quad \pi^c \circ s=id.
$$

Indeed, since $\cH^a, \G^a$ are torsion-free $\Z_\ell$-modules we have
$$
\rH^2(\cH^a, \Z_{\ell}))= \rH^2(\cH^a, \Z_\ell)) \pmod{\ell^n}, \quad \forall n\in \N,
$$ 
and 
$\rH^2(\cH^a, \Z_\ell))$ is a free $\Z_\ell$-module. The groups 
$\G^c,\cH^c$ are determined by the surjective homomorphisms
$$
\wedge^2(\cH^a)\to Z_{\cH}=[\cH^c,\cH^c], \quad \wedge^2(\G^a)\to Z_{\G}=[\G^c,\G^c].
$$
Since $Z_{\cH}, Z_{\G}$ are free $\Z_\ell$-modules, $\Ker(Z_{\G}\to Z_{\cH})$ 
is also a free $\Z_\ell$-module.
\end{proof}

Let $G$ be a finite group, $V$ a faithful representation of $G$ over $k$ and $K=k(V/G)$. 
We have a natural homomorphism $G_K\ra G$. Every valuation $\nu\in \Val_K$ defines a {\em residue} homomorphism
$$
\rH^*_{s}(G,\Z/\ell^n) \hookrightarrow \rH^*(G_K,\Z/\ell^n) \stackrel{\delta_{\nu}}{\lra} \rH^*(G_{K_{\nu}},\Z/\ell^n),
$$ 
and we define the stable 
{\em unramified} cohomology as the kernel of this homomorphism, 
over all divisorial valuations $\nu$:
$$
\rH^*_{s,nr}(G,\Z/\ell^n) = \{ \alpha \in \rH^*_{s}(G,\Z/\ell^n) \,\, \mid \,\, \delta_{\nu}(\alpha)=0 \quad \forall \nu\in \DVal_K \}.
$$
Again, this is independent of the choice of $V$ and is functorial in $G$. 
Fix an element $g\in G$. We say that 
$\alpha\in \rH^*_s(G,\Z/\ell^n)$ is $g$-unramified if 
the restriction of $\alpha $ to the centralizer 
$Z(g)$ of $g$ in $G$ is    
unramified (see \cite{bog-stable} for more details).

\begin{lemm}
\label{lemm:subring}
Let $G$ be a finite group of order coprime to $p=\char(k)$. 
Then 
$$
\rH^*_{s,nr}(G,\Z/\ell^n)\subseteq \rH^*_{s}(G,\Z/\ell^n)
$$ 
is the subring of elements which are $g$-unramified for all $g\in G$.
\end{lemm}

\begin{proof}
We may assume that $G$ is an $\ell$-group, with $\ell$ coprime to $\char(k)$.
By functoriality, a class  $\alpha\in  \rH^*_{s,nr}(G,\Z/\ell^n)$ is also  
$g$-unramified.

Conversely, let $\nu\in \DVal_K$ be a divisorial valuation and $X$ a normal projective
model of $K=k(V/G)$ such that $\nu$ is realized by a divisor  $D\subset X$ and 
both $D, X$ are smooth at the generic point of $D$. Let $D^*$ be a formal neighborhood of this point. 
The map $V\to V/G$ defines a $G$-extension of the completion $K_{\nu}$.
Geometrically, this  
corresponds to a union of finite coverings of formal neighborhoods
of $D^*$, since $G$ has order coprime to $p$: 
the preimage of $D^*$ in $\bar V$ 
is a finite union of smooth formal neighborhoods $D_i^*$
of irreducible divisors $D_i\subset \bar V$.
If the covering $\pi_i: D_i^*\to D$ is unramified at the generic point
of $D_i$ then $\delta_\nu(\alpha)=0$. 
On the other hand, if there is ramification, then
there is a $g\in G$ which acts trivially on some $D_i$, 
and we may assume that $g$ is a generator of a cyclic subgroup acting trivially on $D_i$.  
Consider the subgroup of $G$ which preserves $D_i$ and acts linearly on the
normal bundle of $D_i$. This group is a subgroup of $Z(g)$;
hence there is a $Z(g)$-equivariant map $D_i^*\to V$ for some faithful
linear representation of $Z(g)$ such that $\alpha$ on $D_i^*/Z(g)$ is induced
from $V/Z(g)$.
In particular, if
$\alpha\in \rH^*_{s,nr}(Z(g),\Z/\ell^n)$ then $\delta_\nu(\alpha) =0$.
Thus an element which is unramified for any 
$g\in G$ in $\rH_s^*(G,\Z/\ell^n)$ is unramified.
\end{proof}

The considerations above allow to {\em linearize} the construction of {\em all} finite cohomological 
obstructions to rationality. 

\begin{coro} 
Let 
$$
1\ra Z\ra G^c\ra G^a\ra 1
$$
be a central extension, $g\in G^a$ a nontrivial element, and  $\tilde{g}$ 
a lift of $g$ to $G^c$. 
Then $Z(\tilde{g})$ is a sum of liftable abelian subgroups $\sigma_i$ containing $g$.
\end{coro}

\begin{lemm} 
An element in the image of 
$\rH^*(G^a,\Z/\ell^n) \subset \rH^*_{s,nr}(G^c,\Z/\ell^n)$ 
is $\tilde{g}$-unramified for a primitive element $g$ 
if and only if its restriction to $Z(\tilde{g})$ 
is induced from $Z(\tilde{g})/\langle g\rangle$.
\end{lemm}

\begin{proof} 
One direction is clear. 
Conversely,  $Z(\tilde{g})$ is a central extension of its abelian quotient.
Hence the stabilization homomorphism coincides with the quotient by the ideal $IH_K(n)$ 
(see the proof of Theorem~\ref{thm:bkk}). 
\end{proof}

\begin{coro}
The subring $\rH^*_{s,nr}(G^a,\Z/\ell^n)\subset \rH^*_s(G^a,\Z/\ell^n)$
is defined by $\Sigma$, i.e., by the configuration of liftable subgroups $\sigma_i$.
\end{coro}

%\begin{theo} 
%\label{thm:asok}
%\cite{asok}
%Let $k$ be an algebraically closed field of characteristic zero. For each $n >0$ there exists 
%a smooth projective unirational variety $X$ over $k$ such that 
%$$
%\rH^i_{nr}(X,\mu_r^{\otimes i})=0 \quad \text{ for all } \quad r\in \N, i<n
%$$
%but
%$$
%\rH^n_{nr}(X,\mu_2^{\otimes n})\neq 0.
%$$
%In particular, such $X$ are not rational.  
%\end{theo}

Such cohomological obstructions were considered
by Colliot-Th\'el\`ene and Ojanguren in \cite{ct}, 
where they showed that unramified cohomology 
is an invariant under stable birational equivalence. 
In addition, they produced explicit examples of nontrivial obstructions
in dimension 3. Subsequently, Peyre \cite{peyre1}, \cite{peyre2} 
gave further examples with  $n=3$ and $n=4$
(see also \cite{saltman-h30}, \cite{saltman-h3}).
Similarly to the examples with nontrivial  $\rH^2_{nr}(G)$ in \cite{bog-87}, 
one can construct examples with nontrivial higher cohomology
using as the only input
the combinatorics of the set of liftable subgroups  $\Sigma=\Sigma(G^c)$ for suitable
{\em central extensions} $G^c$.
Since we are interested in function fields $K=k(V/G^c)$ with
trivial $\rH^2_{nr}(K)$, we are looking for groups $G^c$ with
$\rR(G)=\rR_{\wedge}(G)$. Such examples can be found 
by working with analogs of {\em quaternionic} structures on 
linear spaces $G^a=\F_{\ell}^{4n}$, for $n\in \N$.

%Let $G^a$ be a linear space over 
%$\F_\ell$ of dimension $4n$ which
%has also has two structures of $\F_{\ell^2}$ spaces.
%$a,b$ with property  that any $a-\F_{\ell^2}$ intersect
%$b -\F_{\ell^2}$ at most by a $\F_\ell$.
%(Such a group $G^a$ is a finite field analogue of a linear space with quaternionic
%structure).
%Define $R_{\wedge}$ as a subgroup generated by 
%two-dimensional $\F_\ell$-subspaces which are $\F_{\ell^2}$
%$a$-spaces and $b$-spaces.

\section{What about curves?}
\label{sect:curves}

In this section we focus on anabelian geometry of curves over finite fields. By 
Uchida's theorem  (see Theorem~\ref{thm:neukirch}), a curve over 
$k=\F_q$ is uniquely determined by its absolute Galois group. 
Recently, Saidi--Tamagawa proved the {\bf Isom}-version
of Grothendieck's conjecture for the prime-to-characteristic 
geometric fundamental (and absolute Galois) groups
of hyperbolic curves \cite{saidi-tamagawa} 
(generalizing results of Tamagawa and Mochizuki
which dealt with the full groups). 
A {\bf Hom}-form  appears in their recent preprint \cite{saidi-tamagawa09}. 
The authors are interested in {\em rigid} homomorphisms of 
{\em full} and {\rm prime-to-characterstic} 
Galois groups of function fields of curves. 
Modulo passage to open subgroups, a homomorphism
$$
\Psi\colon G_K\ra G_{L}
$$
is called rigid if it preserves the {\em decomposition} subgroups, i.e., 
if for all $\nu\in \DVal_K$
$$
\Psi(D_{\nu})=D_{\nu'},
$$
for some $\nu'\in \DVal_{L}$. 
The main result is that there is a bijection between {\em admissible} homomorphisms 
of fields and rigid homomorphisms of Galois groups
$$
\Hom^{\rm adm}(L, K) \stackrel{\sim}{\lra} \Hom^{\rm rig}(G_K, G_{L})/\sim,
$$ 
modulo conjugation (here {\em admissible} essentially 
means that the extension of function fields $K/L$
is finite of degree coprime to the 
characteristic, see \cite[p. 3]{saidi-tamagawa09}
for a complete description of this notion).    

\

Our work on higher-dimensional anabelian geometry led us to consider homomorphisms
of Galois groups preserving {\em inertia} subgroups.

\begin{theo}
\cite{bt-torelli}
\label{thm:torelli}
Let $K=k(X)$ and $L=l(Y)$ be function fields of curves
over algebraic closures of finite fields. 
Assume that $\mathsf g(X) > 2$ and that 
$$
\Psi \colon G_K^a \ra G^a_L
$$
is an isomorphism of abelianized absolute Galois groups
such that for all $\nu\in \DVal_K$ there exists a $\nu'\in \DVal_L$ with 
$$
\Psi(I^a_{\nu})=I^a_{\nu'}.
$$
Then $k=l$ and the corresponding Jacobians are isogenous.  
\end{theo}

This theorem is a Galois-theoretic incarnation of a finite field 
version of the ``Torelli'' theorem for curves. 
Classically, the setup is as follows: let $k$ be any field and $C/k$ a smooth curve over 
$k$ of genus $\mathsf g(C)\ge 2$, with $C(k)\neq \emptyset$. 
For each $n\in \N$, let $J^n$ be Jacobian of rational equivalence classes of 
degree $n$ zero-cycles on $C$. Put $J^0=J$. 
We have  

\

\centerline{
\xymatrix{
 C^n\ar[r]               & \Sym^n(C)  \ar[r]^{\,\,\,\,\,\,\,\,\lambda_n} &  J^n 
}
}

\

\noindent
Choosing a point $c_0\in C(k)$, we may identify $J^n=J$. 
The image ${\rm Image}(\lambda_{\mathsf g-1})=\Theta\subset J$ is called the theta divisor.
The Torelli theorem asserts that the pair $(J,\Theta)$ determines $C$, up to isomorphism.

\begin{theo} \cite{bt-torelli}
\label{thm:torel}
Let $C,\tilde{C}$ be smooth projective curves of genus 
$\mathsf g \ge 2$ over closures of finite fields $k$ and $\tilde{k}$. 
Let 
$$
\Psi \colon J(k)\stackrel{\sim}{\lra} \tilde{J}(\tilde{k})
$$
be an isomorphism of abelian groups inducing a bijection of sets
$$
C(k)\leftrightarrow \tilde{C}(\tilde{k}).
$$
Then $k=\tilde{k}$ and $J$ is isogenous to $\tilde{J}$. 
\end{theo}

We expect that the curves $C$ and $\tilde{C}$ are isomorphic over $\bar{k}$. 

\

Recall that
$$
J(\bar{\F}_p)= \text{$p$-part}\oplus 
\bigoplus_{\ell\neq p} (\Q_\ell/\Z_{\ell})^{2\mathsf g}.
$$
The main point of Theorem~\ref{thm:torel} is that the 
set $C(\bar{\F}_p)\subset J(\bar{\F}_p)$ {\em rigidifies} this very large torsion abelian group. 
Moreover, we have

\begin{theo} \cite{bt-torelli}
There exists an $N$, bounded effectively in terms of $\mathsf g$, 
such that
$$
\Psi(\Frob)^N \quad \text{ and } \quad 
\widetilde{\Frob}^N
$$
(the respective Frobenius) commute, 
as automorphisms of $\tilde{J}(\tilde{k})$. 
\end{theo}

In some cases, we can prove that the curves $C$ and $\tilde{C}$ are actually isomorphic, 
as algebraic curves. Could Theorem~\ref{thm:torel} hold with $k$ and $\tilde{k}$ replaced by $\C$?
Such an isomorphism $\Psi$ matches all ``special'' points and linear systems 
of the curves. Thus the problem may be amenable to techniques developed in \cite{hru-zilber}, 
where an algebraic curve is reconstructed from an abstract 
``Zariski geometry'' ({\em ibid.}, Proposition 1.1), 
analogously to the reconstruction of projective spaces
from an ``abstract projective geometry'' in Section~\ref{sect:projective}.

\

The proof of Theorem~\ref{thm:torel} has as its 
starting point the following sufficient condition 
for the existence of an isogeny:

\begin{theo}[\cite{bt-torelli}, \cite{bt-tate}] 
\label{thm:ab-divi}
Let $A$ and $\tilde{A}$ be abelian 
varieties of dimension $\mathsf g$ 
over finite fields $k_1$, resp. $\tilde{k}_1$ (of sufficiently divisible cardinality).
Let $k_n/k_1$, resp. $\tilde{k}_n/\tilde{k}_1$, be the unique extensions of degree $n$.
Assume that 
$$
\#A(k_n) \mid \# \tilde{A}(\tilde{k}_n)
$$
for infinitely many $n\in \N$. 
Then ${\rm char}(k)={\rm char}(\tilde{k})$ and 
$A$ and $\tilde{A}$ are isogenous over $\bar{k}$.  
\end{theo}

The proof of this result is based on the theorem of Tate:
$$
{\rm Hom}(A,\tilde{A})\otimes \mathbb Z_{\ell} =
{\rm Hom}_{\mathbb Z_{\ell}[{\rm Fr}]}(T_{\ell}(A),T_{\ell}(\tilde{A}))
$$
and the following, seemingly unrelated, 
theorem concerning divisibilities
of values of {\em recurrence sequences}.

Recall that a {\em linear recurrence} is 
a map $R\,:\, \N\ra \C$ such that
$$
R(n+r) = \sum_{i=0}^{r-1} a_i R(n+i), 
$$
for some $a_i\in \C$ and all $n\in \N$.  
Equivalently, 
\begin{equation}
\label{eqn:Rn}
R(n)=\sum_{\gamma \in \Gamma^0} c_\gamma(n) \gamma^n,
\end{equation}
where $c_\gamma\in \C[x]$ and $\Gamma^0\subset \C^{\times}$ is a finite set of {\em roots} of $R$.
Throughout, we need only {\em simple} recurrences, i.e., those where the {\em characteristic polynomial} 
of $R$ has no multiple roots so that $c_{\gamma}\in \C^{\times}$, for all $\gamma\in \Gamma^0$. 
Let $\Gamma\subset \C^{\times}$ be the group generated by $\Gamma^0$. In our applications 
we may assume that it is torsion-free. Then there is an isomorphism of rings 
$$
\left\{\text{Simple recurrences with roots in }  \Gamma \right\} 
\Leftrightarrow \C[\Gamma], 
$$
where $\C[\Gamma]$ is the ring of Laurent polynomials with exponents in
the finite-rank $\Z$-module $\Gamma$. 
The map 
$$
R\mapsto F_R\in \C[\Gamma]
$$ 
is given by
$$
R\mapsto F_R:= \sum_{\gamma \in \Gamma^0} c_\gamma x^{\gamma}. 
$$

\begin{theo}[Corvaja--Zannier \cite{corvaja-z}]
\label{thm:corv}
Let $R$ and $\tilde{R}$ be simple linear recurrences
such that
\begin{enumerate}
\item $R(n), \tilde{R}(\tilde{n})\neq 0$, for all $n,\tilde{n}\gg 0$;
\item the subgroup $\Gamma\subset \C^{\times}$ generated by
the roots of $R$ and $\tilde{R}$ is torsion-free;
\item there is a finitely-generated subring $\mathfrak A\subset\C$
with $R(n)/\tilde{R}(n)\in \mathfrak A$, for infinitely many $n\in \N$.
\end{enumerate}
Then 
$$
\begin{array}{rcc}
Q\,:\, \N & \ra & \C\\
n         & \mapsto &  R(n)/\tilde{R}(n)
\end{array}
$$ 
is a simple linear recurrence. 
In particular, $F_Q\in \C[\Gamma]$ and 
$$
F_Q \cdot F_{\tilde{R}}=F_R.  
$$
\end{theo}

This very useful theorem concerning divisibilities
is actually an application of a known case of the Lang--Vojta conjecture
concerning nondensity of integral points on 
``hyperbolic''  varieties, i.e., quasi-projective varieties of log-general type. 
In this case, one is interested in subvarieties of algebraic tori and the needed result is 
{\em Schmidt's subspace theorem}. Other applications of this result to integral points and
diophantine approximation are discussed in \cite{bilu}, 
and connections to Vojta's conjecture in  \cite{silverman1}, \cite{silverman}.

A rich source of interesting simple linear recurrences is geometry over finite fields. 
Let $X$ be a  smooth projective variety over $k_1=\mathbb{F}_q$ of dimension $d$, 
$\bar{X}=X\times_{k_1}\bar{k}_1$,
and let $k_n/k_1$ be the unique extension of degree $n$. 
Then 
$$
\# X(k_n):= \tr(\Fr^n) =\sum_{i=0}^{2d} (-1)^ic_{ij} \alpha_{ij}^n,
$$
where $\Fr$ is Frobenius acting on \'etale cohomology 
$\rH^*_{et}(\bar{X},\Q_{\ell})$, with $\ell\nmid q$, and $c_{ij}\in \C^{\times}$. 
Let $\Gamma^0:=\{\alpha_{ij}\}$ be the set of corresponding eigenvalues.
and $\Gamma_X\subset \C^{\times}$ the multiplicative group
generated by $\alpha_{ij}$. It is torsion-free provided 
the cardinality of $k_1$ is sufficiently divisible.

For example, let $A$ be an abelian variety over $k_1$, 
$\{\alpha_j\}_{j=1,\ldots, 2\mathsf g}$ the set of eigenvalues
of the Frobenius on $\rH^1_{et}(\bar{A}, \Q_{\ell})$, for $\ell\neq p$, 
and $\Gamma_A\subset \mathbb C^{\times}$ the multiplicative subgroup spanned by the $\alpha_j$. 
Then
\begin{equation}
\label{eqn:rnn}
R(n):=\#A(k_n) = \prod_{j=1}^{2\mathsf g} (\alpha_j^n-1).
\end{equation}
is a simple linear recurrence with roots in $\Gamma_A$. 
Theorem~\ref{thm:ab-divi} follows by applying Theorem~\ref{thm:corv} to 
this recurrence and exploiting the special shape of the  
Laurent polynomial associated to \eqref{eqn:rnn}. 

 \

We now sketch a proof of Theorem~\ref{thm:torel}, assuming for simplicity that 
$C$ be a nonhyperelliptic curve of genus $\mathsf g(C)\ge 3$. 

\

{\em Step 1.}
For all finite fields $k_1$ with sufficiently many elements ($\ge c\mathsf g^2$) 
the group $J(k_1)$ is generated by $C(k_1)$, by \cite[Corollary 5.3]{bt-torelli}.
Let 
$$
k_1\subset k_2\subset \ldots \subset k_n\subset \ldots
$$
be the tower of degree 2 extensions. 
To \emph{characterize} $J(k_n)$ it suffices to 
characterize $C(k_n)$.

\

{\em Step 2.} For each $n\in \N$, the abelian group
$J(k_{n})$ is generated by $c \in C(k)$ such that
there exists a point $c'\in C(k)$ with 
$$
c+c'\in J(k_{n-1}).
$$

\

{\em Step 3.}
Choose $k_1,\tilde{k_1}$ (sufficiently large) 
such that 
$$
\Psi(J(k_1))\subset \tilde{J}(\tilde{k}_1)
$$
Define $C(k_n)$, resp. $\tilde{C}(\tilde{k}_n)$, intrinsically, using only the group- and set-theoretic
information as above.
Then 
$$
\Psi(J(k_n))\subset \tilde{J}(\tilde{k}_n), \quad  \text{ for all } \quad n\in \N.
$$
and
$$
\# J(k_n) \mid \# \tilde{J}(\tilde{k}_n).
$$
To conclude the proof of Theorem~\ref{thm:torel} 
it suffices to apply Theorem~\ref{thm:corv} and Theorem~\ref{thm:ab-divi} 
about divisibility of recurrence sequences.

\
 
One of the strongest and somewhat counter-intuitive results in this area 
is a theorem of Tamagawa:

\begin{theo}\cite{tamagawa-pi}
There are at most finitely many (isomorphism classes of) 
curves of genus $\mathsf g$ over $k=\bar{\F}_p$ 
with given (profinite) geometric fundamental group. 
\end{theo}

On the other hand, in 2002 we proved:

\begin{theo} \cite{bt-unram}
Let $C$ be a hyperelliptic curve of genus $\ge 2$  over 
$k=\bar{\F}_p$, with $p\ge 5$. Then  
for every curve $C'$ over $k$ there exists an \'etale cover
$\pi\colon\tilde{C}\ra C$ and surjective map $\tilde{C}\ra C'$. 
\end{theo}

This shows that the geometric fundamental groups of 
hyperbolic curves are ``almost'' independent of the curve:
every such $\pi_1(C)$ has a subgroup of small index and such that the quotient by this
subgroup is almost abelian, surjecting onto 
the fundamental group of another curve $C'$.

This relates to the problem of couniformization for
hyperbolic curves (see \cite{bt-unram}). 
The Riemann theorem says that
the unit disc in the complex plane 
serves as a universal covering for all 
complex projective curves of genus $\geq 2$, simultaneously. 
This provides a  canonical embedding 
of the fundamental group of a curve into
the group of complex automorphisms of the disc, which is isomorphic to $\PGL_2(\R)$.
In particular, 
it defines a natural embedding of the 
field of rational functions on the curve into 
the field of meromorphic functions on the disc.
The latter is unfortunately is too large to be of any help in
solving concrete problems.

However, in some cases there is an algebraic
substitute. For example, in the class of modular curves
there is a natural pro-algebraic object $Mod$ 
(introduced by Shafarevich) 
which is given by a tower of modular curves; 
the corresponding pro-algebraic field,
which is an inductive union $M$ of the fields of rational functions on
modular curves. Similarly to the case of a disc the space $Mod$ has
a wealth of of symmetries which contains a product  $\prod_p \SL_2(\Z_p)$
and the absolute Galois group $\Gal(\bar{\Q}/\Q)$.

The above result alludes to the existence of 
a similar disc-type algebraic object
for all hyperbolic curves defined over $\bar{\F}_p$
(or even for arithmetic hyperbolic curves).

For example consider $C_6$ given by 
$y^6 = x(x-1)$ over $\F_p$, with $p\neq 2,3$,
and define $\tilde{C}_6$ as a pro-algebraic universal covering
of $C_6$. Thus $\bar{\F}_p(\tilde{C}_6) = \bigcup \bar{\F}_p(C_i)$,
where $C_i$ range over all finite geometrically nonramified
coverings of $C_6$. Then $\bar{\F}_p(\tilde{C}_6)$ contains all
other fields $\bar{\F}_p (C)$, where $C$ is an arbitrary curve
defined over some $ \F_q\subset  \bar{\F}_p $. Note that it also implies that
\'etale fundamental group $\pi_1(C_6)$ contains a subgroup of finite index
which surjects onto $\pi_1(C)$ with the action of
$\hat{\Z} = \Gal(\bar{\F}_p/\F_q)$.

The corresponding results in the case of curves over
number fields $K\subset \bar{\Q}$ 
are weaker, but even in the weak form they are quite intriguing.

\bibliographystyle{smfplain}
\bibliography{msri}

\def\cftil#1{\ifmmode\setbox7\hbox{$\accent"5E#1$}\else
  \setbox7\hbox{\accent"5E#1}\penalty 10000\relax\fi\raise 1\ht7
  \hbox{\lower1.15ex\hbox to 1\wd7{\hss\accent"7E\hss}}\penalty 10000
  \hskip-1\wd7\penalty 10000\box7}
\providecommand{\bysame}{\leavevmode ---\ }
\providecommand{\og}{``}
\providecommand{\fg}{''}
\providecommand{\smfandname}{\&}
\providecommand{\smfedsname}{\'eds.}
\providecommand{\smfedname}{\'ed.}
\providecommand{\smfmastersthesisname}{M\'emoire}
\providecommand{\smfphdthesisname}{Th\`ese}
\begin{thebibliography}{10}

\bibitem{bilu}
{\scshape Y.~F. Bilu} -- {\og The many faces of the subspace theorem [after
  {A}damczewski, {B}ugeaud, {C}orvaja, {Z}annier{$\ldots$}]\fg},
  \emph{Ast\'erisque} (2008), no.~317, p.~Exp. No. 967, vii, 1--38,
  S{\'e}minaire Bourbaki. Vol. 2006/2007.

\bibitem{bog-87}
{\scshape F.~Bogomolov} -- {\og The {B}rauer group of quotient spaces of linear
  representations\fg}, \emph{Izv. Akad. Nauk SSSR Ser. Mat.} \textbf{51}
  (1987), no.~3, p.~485--516, 688.

\bibitem{B-1}
\bysame , {\og Abelian subgroups of {G}alois groups\fg}, \emph{Izv. Akad. Nauk
  SSSR Ser. Mat.} \textbf{55} (1991), no.~1, p.~32--67.

\bibitem{B-3}
\bysame , {\og On two conjectures in birational algebraic geometry\fg}, in
  \emph{Algebraic geometry and analytic geometry (Tokyo, 1990)}, ICM-90 Satell.
  Conf. Proc., Springer, Tokyo, 1991, p.~26--52.

\bibitem{bog-stable}
\bysame , {\og Stable cohomology of groups and algebraic varieties\fg},
  \emph{Mat. Sb.} \textbf{183} (1992), no.~5, p.~3--28.

\bibitem{bog-linear}
\bysame , {\og Linear tori with an action of finite groups\fg}, \emph{Mat.
  Zametki} \textbf{57} (1995), no.~5, p.~643--652, 796.

\bibitem{B-2}
\bysame , {\og On the structure of {G}alois groups of the fields of rational
  functions\fg}, in \emph{$K$-theory and algebraic geometry: connections with
  quadratic forms and division algebras (Santa Barbara, CA, 1992)}, Proc.
  Sympos. Pure Math., vol.~58, Amer. Math. Soc., Providence, RI, 1995,
  p.~83--88.

\bibitem{bt-torelli}
{\scshape F.~Bogomolov, M.~Korotiaev {\normalfont \smfandname} Y.~Tschinkel} --
  {\og A {T}orelli theorem for curves over finite fields\fg}, \emph{Pure and
  Applied Math Quarterly, Tate Festschrift} \textbf{6} (2010), no.~1,
  p.~245--294.

\bibitem{BT}
{\scshape F.~Bogomolov {\normalfont \smfandname} Y.~Tschinkel} -- {\og
  Commuting elements in {G}alois groups of function fields\fg}, in
  \emph{Motives, Polylogarithms and Hodge theory}, International Press, 2002,
  p.~75--120.

\bibitem{bt0}
\bysame , {\og Reconstruction of function fields\fg}, \emph{Geom. Funct. Anal.}
  \textbf{18} (2008), no.~2, p.~400--462.

\bibitem{bt-milnor}
\bysame , {\og Milnor ${K}_2$ and field homomorphisms\fg}, in \emph{Surveys in
  Differential Geometry XIII}, International Press, 2009, p.~223--244.

\bibitem{bt1}
\bysame , {\og Reconstruction of higher-dimensional function fields\fg}, 2009,
  {\tt arXiv:0912.4923}.

\bibitem{bog-unram}
{\scshape F.~Bogomolov, J.~Maciel {\normalfont \smfandname} T.~Petrov} -- {\og
  Unramified {B}rauer groups of finite simple groups of {L}ie type {$A_l$}\fg},
  \emph{Amer. J. Math.} \textbf{126} (2004), no.~4, p.~935--949.

\bibitem{bt-unram}
{\scshape F.~Bogomolov {\normalfont \smfandname} Y.~Tschinkel} -- {\og
  Unramified correspondences\fg}, in \emph{Algebraic number theory and
  algebraic geometry}, Contemp. Math., vol. 300, Amer. Math. Soc., Providence,
  RI, 2002, p.~17--25.

\bibitem{bt-tate}
\bysame , {\og On a theorem of {T}ate\fg}, \emph{Cent. Eur. J. Math.}
  \textbf{6} (2008), no.~3, p.~343--350.

\bibitem{efrat-minac}
{\scshape S.~K. Chebolu, I.~Efrat {\normalfont \smfandname} J.~Min{\'a}{\v{c}}}
  -- {\og Quotients of absolute {G}alois groups which determine the entire
  {G}alois cohomology\fg}, 2009, {\tt arXiv:0905.1364}.

\bibitem{chebolu-minac}
{\scshape S.~K. Chebolu {\normalfont \smfandname} J.~Min{\'a}{\v{c}}} -- {\og
  Absolute {G}alois groups viewed from small quotients and the {B}loch-{K}ato
  conjecture\fg}, in \emph{New topological contexts for {G}alois theory and
  algebraic geometry ({BIRS} 2008)}, Geom. Topol. Monogr., vol.~16, Geom.
  Topol. Publ., Coventry, 2009, p.~31--47.

\bibitem{ct}
{\scshape J.-L. Colliot-Th{\'e}l{\`e}ne {\normalfont \smfandname} M.~Ojanguren}
  -- {\og Vari\'et\'es unirationnelles non rationnelles: au-del\`a de l'exemple
  d'{A}rtin et {M}umford\fg}, \emph{Invent. Math.} \textbf{97} (1989), no.~1,
  p.~141--158.

\bibitem{corry-pop}
{\scshape S.~Corry {\normalfont \smfandname} F.~Pop} -- {\og The pro-{$p$}
  {H}om-form of the birational anabelian conjecture\fg}, \emph{J. Reine Angew.
  Math.} \textbf{628} (2009), p.~121--127.

\bibitem{corvaja-z}
{\scshape P.~Corvaja {\normalfont \smfandname} U.~Zannier} -- {\og Finiteness
  of integral values for the ratio of two linear recurrences\fg}, \emph{Invent.
  Math.} \textbf{149} (2002), no.~2, p.~431--451.

\bibitem{efrat}
{\scshape I.~Efrat} -- {\og Construction of valuations from {$K$}-theory\fg},
  \emph{Math. Res. Lett.} \textbf{6} (1999), no.~3-4, p.~335--343.

\bibitem{engler-koe}
{\scshape A.~J. Engler {\normalfont \smfandname} J.~Koenigsmann} -- {\og
  Abelian subgroups of pro-{$p$} {G}alois groups\fg}, \emph{Trans. Amer. Math.
  Soc.} \textbf{350} (1998), no.~6, p.~2473--2485.

\bibitem{engler-nogu}
{\scshape A.~J. Engler {\normalfont \smfandname} J.~B. Nogueira} -- {\og
  Maximal abelian normal subgroups of {G}alois pro-{$2$}-groups\fg}, \emph{J.
  Algebra} \textbf{166} (1994), no.~3, p.~481--505.

\bibitem{eh1}
{\scshape D.~M. Evans {\normalfont \smfandname} E.~Hrushovski} -- {\og
  Projective planes in algebraically closed fields\fg}, \emph{Proc. London
  Math. Soc. (3)} \textbf{62} (1991), no.~1, p.~1--24.

\bibitem{eh2}
\bysame , {\og The automorphism group of the combinatorial geometry of an
  algebraically closed field\fg}, \emph{J. London Math. Soc. (2)} \textbf{52}
  (1995), no.~2, p.~209--225.

\bibitem{faltings-tate}
{\scshape G.~Faltings} -- {\og Endlichkeitss\"atze f\"ur abelsche
  {V}ariet\"aten \"uber {Z}ahlk\"orpern\fg}, \emph{Invent. Math.} \textbf{73}
  (1983), no.~3, p.~349--366.

\bibitem{faltings}
{\scshape G.~Faltings} -- {\og Curves and their fundamental groups (following
  {G}rothendieck, {T}amagawa and {M}ochizuki)\fg}, \emph{Ast\'erisque} (1998),
  no.~252, p.~Exp.\ No.\ 840, 4, 131--150, S{\'e}minaire Bourbaki. Vol.
  1997/98.

\bibitem{gism}
{\scshape J.~Gismatullin} -- {\og Combinatorial geometries of field
  extensions\fg}, \emph{Bull. Lond. Math. Soc.} \textbf{40} (2008), no.~5,
  p.~789--800.

\bibitem{groth-letter}
{\scshape A.~Grothendieck} -- {\og Brief an {G}. {F}altings\fg}, in
  \emph{Geometric {G}alois actions, 1}, London Math. Soc. Lecture Note Ser.,
  vol. 242, Cambridge Univ. Press, Cambridge, 1997, With an English translation
  on pp. 285--293, p.~49--58.

\bibitem{Weibel-1}
{\scshape C.~Haesemeyer {\normalfont \smfandname} C.~Weibel} -- {\og Norm
  varieties and the chain lemma (after {M}arkus {R}ost)\fg}, in \emph{Algebraic
  topology}, Abel Symp., vol.~4, Springer, Berlin, 2009, p.~95--130.

\bibitem{hru-zilber}
{\scshape E.~Hrushovski {\normalfont \smfandname} B.~Zilber} -- {\og Zariski
  geometries\fg}, \emph{J. Amer. Math. Soc.} \textbf{9} (1996), no.~1,
  p.~1--56.

\bibitem{ihara-naka}
{\scshape Y.~Ihara {\normalfont \smfandname} H.~Nakamura} -- {\og Some
  illustrative examples for anabelian geometry in high dimensions\fg}, in
  \emph{Geometric {G}alois actions, 1}, London Math. Soc. Lecture Note Ser.,
  vol. 242, Cambridge Univ. Press, Cambridge, 1997, p.~127--138.

\bibitem{koenigsmann}
{\scshape J.~Koenigsmann} -- {\og On the `section conjecture' in anabelian
  geometry\fg}, \emph{J. Reine Angew. Math.} \textbf{588} (2005), p.~221--235.

\bibitem{kun}
{\scshape B.~Kunyavski} -- {\og The {B}ogomolov multiplier of finite simple
  groups\fg}, in \emph{Cohomological and geometric approaches to rationality
  problems}, Progr. in Math., vol. 282, Birkh\"auser, Basel, 2010, p.~209--217.

\bibitem{mochizuki}
{\scshape S.~Mochizuki} -- {\og The local pro-{$p$} anabelian geometry of
  curves\fg}, \emph{Invent. Math.} \textbf{138} (1999), no.~2, p.~319--423.

\bibitem{mochi-topics}
\bysame , {\og Topics surrounding the anabelian geometry of hyperbolic
  curves\fg}, in \emph{Galois groups and fundamental groups}, Math. Sci. Res.
  Inst. Publ., vol.~41, Cambridge Univ. Press, Cambridge, 2003, p.~119--165.

\bibitem{N}
{\scshape H.~Nakamura} -- {\og Galois rigidity of the \'etale fundamental
  groups of punctured projective lines\fg}, \emph{J. Reine Angew. Math.}
  \textbf{411} (1990), p.~205--216.

\bibitem{Na-Mo}
{\scshape H.~Nakamura, A.~Tamagawa {\normalfont \smfandname} S.~Mochizuki} --
  {\og Grothendieck's conjectures concerning fundamental groups of algebraic
  curves\fg}, \emph{S\=ugaku} \textbf{50} (1998), no.~2, p.~113--129.

\bibitem{neukirch}
{\scshape J.~Neukirch} -- {\og Kennzeichnung der {$p$}-adischen und der
  endlichen algebraischen {Z}ahlk\"orper\fg}, \emph{Invent. Math.} \textbf{6}
  (1969), p.~296--314.

\bibitem{peyre1}
{\scshape E.~Peyre} -- {\og Unramified cohomology and rationality problems\fg},
  \emph{Math. Ann.} \textbf{296} (1993), no.~2, p.~247--268.

\bibitem{peyre2}
\bysame , {\og Unramified cohomology of degree 3 and {N}oether's problem\fg},
  \emph{Invent. Math.} \textbf{171} (2008), no.~1, p.~191--225.

\bibitem{pop-unpub}
{\scshape F.~Pop} -- {\og Pro-$\ell$ birational anabelian geometry over
  algebraically closed fields {I}\fg}, {\tt arXiv:math/0307076}, unpublished.

\bibitem{pop}
\bysame , {\og On {G}rothendieck's conjecture of birational anabelian
  geometry\fg}, \emph{Ann. of Math. (2)} \textbf{139} (1994), no.~1,
  p.~145--182.

\bibitem{P-2}
\bysame , {\og Glimpses of {G}rothendieck's anabelian geometry\fg}, in
  \emph{Geometric Galois actions, 1}, London Math. Soc. Lecture Note Ser., vol.
  242, Cambridge Univ. Press, Cambridge, 1997, p.~113--126.

\bibitem{pop-alter}
\bysame , {\og Alterations and birational anabelian geometry\fg}, in
  \emph{Resolution of singularities ({O}bergurgl, 1997)}, Progr. Math., vol.
  181, Birkh\"auser, Basel, 2000, p.~519--532.

\bibitem{posi}
{\scshape L.~Positselski} -- {\og Koszul property and {B}ogomolov's
  conjecture\fg}, \emph{Int. Math. Res. Not.} (2005), no.~31, p.~1901--1936.

\bibitem{posi-vish}
{\scshape L.~Positselski {\normalfont \smfandname} A.~Vishik} -- {\og Koszul
  duality and {G}alois cohomology\fg}, \emph{Math. Res. Lett.} \textbf{2}
  (1995), no.~6, p.~771--781.

\bibitem{saidi-tamagawa09}
{\scshape M.~Sa{\"{\i}}di {\normalfont \smfandname} A.~Tamagawa} -- {\og On the
  {H}om-form {G}rothendieck's birational anabelian conjecture in characteristic
  $p>0$\fg}, 2009, {\tt arXiv:0912.1972}.

\bibitem{saidi-tamagawa}
\bysame , {\og A prime-to-{$p$} version of {G}rothendieck's anabelian
  conjecture for hyperbolic curves over finite fields of characteristic
  {$p>0$}\fg}, \emph{Publ. Res. Inst. Math. Sci.} \textbf{45} (2009), no.~1,
  p.~135--186.

\bibitem{saltman}
{\scshape D.~J. Saltman} -- {\og Noether's problem over an algebraically closed
  field\fg}, \emph{Invent. Math.} \textbf{77} (1984), no.~1, p.~71--84.

\bibitem{saltman-torus}
\bysame , {\og Multiplicative field invariants\fg}, \emph{J. Algebra}
  \textbf{106} (1987), no.~1, p.~221--238.

\bibitem{saltman-h30}
\bysame , {\og Brauer groups of invariant fields, geometrically negligible
  classes, an equivariant {C}how group, and unramified {$H^3$}\fg}, in
  \emph{{$K$}-theory and algebraic geometry: connections with quadratic forms
  and division algebras ({S}anta {B}arbara, {CA}, 1992)}, Proc. Sympos. Pure
  Math., vol.~58, Amer. Math. Soc., Providence, RI, 1995, p.~189--246.

\bibitem{saltman-h3}
\bysame , {\og {$H^3$} and generic matrices\fg}, \emph{J. Algebra} \textbf{195}
  (1997), no.~2, p.~387--422.

\bibitem{silverman1}
{\scshape J.~H. Silverman} -- {\og Generalized greatest common divisors,
  divisibility sequences, and {V}ojta's conjecture for blowups\fg},
  \emph{Monatsh. Math.} \textbf{145} (2005), no.~4, p.~333--350.

\bibitem{silverman}
\bysame , {\og Greatest common divisors and algebraic geometry\fg}, in
  \emph{Diophantine geometry}, CRM Series, vol.~4, Ed. Norm., Pisa, 2007,
  p.~297--308.

\bibitem{suslin-ob}
{\scshape A.~A. Suslin} -- {\og Algebraic {$K$}-theory and the norm residue
  homomorphism\fg}, in \emph{Current problems in mathematics, {V}ol. 25}, Itogi
  Nauki i Tekhniki, Akad. Nauk SSSR Vsesoyuz. Inst. Nauchn. i Tekhn. Inform.,
  Moscow, 1984, p.~115--207.

\bibitem{T}
{\scshape A.~Tamagawa} -- {\og The {G}rothendieck conjecture for affine
  curves\fg}, \emph{Compositio Math.} \textbf{109} (1997), no.~2, p.~135--194.

\bibitem{tamagawa-pi}
\bysame , {\og Finiteness of isomorphism classes of curves in positive
  characteristic with prescribed fundamental groups\fg}, \emph{J. Algebraic
  Geom.} \textbf{13} (2004), no.~4, p.~675--724.

\bibitem{tate}
{\scshape J.~Tate} -- {\og Endomorphisms of abelian varieties over finite
  fields\fg}, \emph{Invent. Math.} \textbf{2} (1966), p.~134--144.

\bibitem{uchida}
{\scshape K.~Uchida} -- {\og Isomorphisms of {G}alois groups of algebraic
  function fields\fg}, \emph{Ann. Math. (2)} \textbf{106} (1977), no.~3,
  p.~589--598.

\bibitem{V91}
{\scshape V.~A. Voevodski{\u\i}} -- {\og Galois groups of function fields over
  fields of finite type over {${\bf Q}$}\fg}, \emph{Uspekhi Mat. Nauk}
  \textbf{46} (1991), no.~5(281), p.~163--164.

\bibitem{V2}
\bysame , {\og Galois representations connected with hyperbolic curves\fg},
  \emph{Izv. Akad. Nauk SSSR Ser. Mat.} \textbf{55} (1991), no.~6,
  p.~1331--1342.

\bibitem{MC-2}
{\scshape V.~Voevodsky} -- {\og Reduced power operations in motivic
  cohomology\fg}, \emph{Publ. Math. Inst. Hautes \'Etudes Sci.} (2003), no.~98,
  p.~1--57.

\bibitem{MC-l}
\bysame , {\og On motivic cohomology with $\mathbb z/\ell$-coefficients\fg},
  2010, {\em Annals of Math.}, to appear.

\bibitem{Weibel-2}
{\scshape C.~Weibel} -- {\og The norm residue isomorphism theorem\fg}, \emph{J.
  Topol.} \textbf{2} (2009), no.~2, p.~346--372.

\end{thebibliography}

\end{document}